\documentclass[12pt]{amsart}
\usepackage{}
\usepackage{amsmath}
\usepackage{amsfonts}
\usepackage{amssymb}
\usepackage[all,cmtip]{xy}           %xypic macro for latex2.09
\usepackage{bbding}
\usepackage{txfonts}
\usepackage[shortlabels]{enumitem}
\usepackage{ifpdf}
\ifpdf
\usepackage[colorlinks,final,backref=page,hyperindex]{hyperref}
\else
\usepackage[colorlinks,final,backref=page,hyperindex,hypertex]{hyperref}
\fi
\usepackage{tikz}
\usepackage[active]{srcltx}
\usepackage{graphicx}
\usepackage{bm}
\usepackage{url}
\usepackage{bbm}
\usepackage{diagbox}
\usepackage{makecell}
\usepackage{caption}
\usepackage{longtable}
\usepackage{makecell}
\usepackage{multirow}
\usepackage{booktabs}
\usepackage{caption}
\usepackage{threeparttable}
\usepackage{graphicx}
\newcolumntype{L}{>{$}l<{$}}
\newcolumntype{R}{>{$}r<{$}}
\usepackage{setspace}
\usepackage{csquotes}

\makeatletter

\newtheorem{theorem}{Theorem}[section]
\newtheorem{prop}[theorem]{Proposition}

\newtheorem{prop-def}{Proposition-Definition}[section]

\theoremstyle{definition}
\newtheorem{defn}[theorem]{Definition}

\newtheorem{remark}[theorem]{Remark}

\newcommand{\nc}{\newcommand}

\newcommand {\emptycomment}[1]{}

\nc{\delete}[1]{{}}
\nc{\mmargin}[1]{}

\nc{\mlabel}[1]{\label{#1}}  % Use this to suppress names
\nc{\mcite}[1]{\cite{#1}}  % Use this to suppress names
\nc{\mref}[1]{\ref{#1}}  % Use this to suppress names
\nc{\meqref}[1]{\eqref{#1}}  % Use this to suppress names
\nc{\mbibitem}[1]{\bibitem{#1}} % Use this to show number
\delete{
	\nc{\mlabel}[1]{\label{#1}  % Use the next two lines to show names
		{\hfill \hspace{1cm}{\bf{{\ }\hfill(#1)}}}}
	\nc{\mcite}[1]{\cite{#1}{{\bf{{\ }(#1)}}}}  % Use this lines to show names
	\nc{\mref}[1]{\ref{#1}{{\bf{{\ }(#1)}}}}  % Use this lines to show names
	\nc{\meqref}[1]{\eqref{#1}{{\bf{{\ }(#1)}}}}  % Use this lines to show names
	\nc{\mbibitem}[1]{\bibitem[\bf #1]{#1}} % Use this to show name
}

 \font\cyrs=wncyr7

\nc{\vep}{\varepsilon}
\nc{\bin}[2]{ (_{\stackrel{\scs{#1}}{\scs{#2}}})}  %binomial coeff
\nc{\binc}[2]{(\!\! \begin{array}{c} \scs{#1}\\
		\scs{#2} \end{array}\!\!)}  %binomial coeff
\nc{\bincc}[2]{  ( {\scs{#1} \atop
		\vspace{-1cm}\scs{#2}} )}  %binomial coeff
\nc{\oline}[1]{\overline{#1}}
\nc{\mapm}[1]{\lfloor\!|{#1}|\!\rfloor}
\nc{\bs}{\bar{S}}
\nc{\cast}{{\,\mbox{\raisebox{.8pt}{$\scriptstyle \circledast$}}\,}}
\nc{\la}{\longrightarrow}
\nc{\ot}{\otimes}
\nc{\rar}{\rightarrow}
\nc{\dar}{\downarrow}
\nc{\dap}[1]{\downarrow \rlap{$\scriptstyle{#1}$}}
\nc{\defeq}{\stackrel{\rm def}{=}}
\nc{\dis}[1]{\displaystyle{#1}}
\nc{\dotcup}{\ \displaystyle{\bigcup^\bullet}\ }
\nc{\hcm}{\ \hat{,}\ }
\nc{\hts}{\hat{\otimes}}
\nc{\hcirc}{\hat{\circ}}
\nc{\lleft}{[}
\nc{\lright}{]}
\nc{\curlyl}{\left \{ \begin{array}{c} {} \\ {} \end{array}
	\right .  \!\!\!\!\!\!\!}
\nc{\curlyr}{ \!\!\!\!\!\!\!
	\left . \begin{array}{c} {} \\ {} \end{array}
	\right \} }
\nc{\longmid}{\left | \begin{array}{c} {} \\ {} \end{array}
	\right . \!\!\!\!\!\!\!}
\nc{\ora}[1]{\stackrel{#1}{\rar}}
\nc{\ola}[1]{\stackrel{#1}{\la}}%${\Bbb Z}$
\nc{\scs}[1]{\scriptstyle{#1}} \nc{\mrm}[1]{{\rm #1}}
\nc{\dirlim}{\displaystyle{\lim_{\longrightarrow}}\,}
\nc{\invlim}{\displaystyle{\lim_{\longleftarrow}}\,}
\nc{\dislim}[1]{\displaystyle{\lim_{#1}}} \nc{\colim}{\mrm{colim}}
\nc{\mvp}{\vspace{0.3cm}} \nc{\tk}{^{(k)}} \nc{\tp}{^\prime}
\nc{\ttp}{^{\prime\prime}} \nc{\svp}{\vspace{2cm}}
\nc{\vp}{\vspace{8cm}}
\nc{\modg}[1]{\!<\!\!{#1}\!\!>}
\nc{\intg}[1]{F_C(#1)}
\nc{\lmodg}{\!<\!\!}
\nc{\rmodg}{\!\!>\!}
\nc{\cpi}{\widehat{\Pi}}
\nc{\ssha}{{\mbox{\cyrs X}}} %sha as product
\nc{\tsha}{{\mbox{\cyrt X}}}
\nc{\shpr}{\diamond}    %Shuffle product
\nc{\labs}{\mid\!}
\nc{\rabs}{\!\mid}
\nc{\btr}{\blacktriangleright}

\nc{\ad}{\mrm{ad}}
\nc{\rRB}{\mathsf{rRB}}
\nc{\cocrRB}{\mathsf{cocrRB}}
\nc{\PH}{\mathsf{PH}}
\nc{\cocPH}{\mathsf{cocPH}}
\nc{\ann}{\mrm{ann}}
\nc{\Aut}{\mrm{Aut}}
\nc{\Der}{\mrm{Der}}
\nc{\Sym}{\mrm{Sym}}
\nc{\br}{\mrm{bre}}
\nc{\can}{\mrm{can}}
\nc{\Cont}{\mrm{Cont}}
\nc{\rchar}{\mrm{char}}
\nc{\cok}{\mrm{coker}}
\nc{\de}{\mrm{dep}}
\nc{\dtf}{{R-{\rm tf}}}
\nc{\dtor}{{R-{\rm tor}}}

\nc{\Dif}{\mrm{Diff}}
\nc{\Div}{\mrm{Div}}
\nc{\End}{\mrm{End}}
\nc{\Ext}{\mrm{Ext}}
\nc{\Fil}{\mrm{Fil}}
\nc{\Fr}{\mrm{Fr}}
\nc{\Frob}{\mrm{Frob}}
\nc{\Gal}{\mrm{Gal}}
\nc{\GL}{\mrm{GL}}
\nc{\Gr}{\mrm{Gr}}
\nc{\Hom}{\mrm{Hom}}
\nc{\Hoch}{\mrm{Hoch}}
\nc{\hsr}{\mrm{H}}
\nc{\hpol}{\mrm{HP}}
\nc{\id}{\mrm{id}}
\nc{\im}{\mrm{im}}
\nc{\inv}{\mrm{inv}}
\nc{\Id}{\mrm{Id}}
\nc{\ID}{\mrm{ID}}
\nc{\Irr}{\mrm{Irr}}
\nc{\incl}{\mrm{incl}}
\nc{\length}{\mrm{length}}
\nc{\NLSW}{\mrm{NLSW}}
\nc{\Lie}{\mrm{Lie}}
\nc{\mchar}{\rm char}
\nc{\mpart}{\mrm{part}}
\nc{\ql}{{\QQ_\ell}}
\nc{\qp}{{\QQ_p}}
\nc{\rank}{\mrm{rank}}
\nc{\rcot}{\mrm{cot}}
\nc{\rdef}{\mrm{def}}
\nc{\rdiv}{{\rm div}}
\nc{\rtf}{{\rm tf}}
\nc{\rtor}{{\rm tor}}
\nc{\res}{\mrm{res}}
\nc{\SL}{\mrm{SL}}
\nc{\Spec}{\mrm{Spec}}
\nc{\tor}{\mrm{tor}}
\nc{\Tr}{\mrm{Tr}}
\nc{\tr}{\mrm{tr}}
\nc{\wt}{\mrm{wt}}

\nc{\bfk}{{\bf k}}
\nc{\bfone}{{\bf 1}}
\nc{\bfzero}{{\bf 0}}
\nc{\detail}{\marginpar{\bf More detail}
	\noindent{\bf Need more detail!}
	\svp}
\nc{\gap}{\marginpar{\bf Incomplete}\noindent{\bf Incomplete!!}
	\svp}
\nc{\FMod}{\mathbf{FMod}}
\nc{\Int}{\mathbf{Int}}
\nc{\Mon}{\mathbf{Mon}}
\nc{\remarks}{\noindent{\bf Remarks: }}
\nc{\Rep}{\mathbf{Rep}}
\nc{\Rings}{\mathbf{Rings}}
\nc{\Sets}{\mathbf{Sets}}
\nc{\Diff}{\mathbf{Diff}}
\nc{\Inte}{\mathbf{Inte}}
\nc{\U}{\mathbf{U}}

\nc{\BA}{{\mathbb A}}   \nc{\CC}{{\mathbb C}}
\nc{\DD}{{\mathbb D}}   \nc{\EE}{{\mathbb E}}
\nc{\FF}{{\mathbb F}}   \nc{\GG}{{\mathbb G}}
\nc{\HH}{{\mathbb H}}   \nc{\LL}{{\mathbb L}}
\nc{\NN}{{\mathbb N}}   \nc{\PP}{{\mathbb P}}
\nc{\QQ}{{\mathbb Q}}   \nc{\RR}{{\mathbb R}}
\nc{\TT}{{\mathbb T}}   \nc{\VV}{{\mathbb V}}
\nc{\ZZ}{{\mathbb Z}}   \nc{\TP}{\widetilde{P}}

\nc{\cala}{{\mathcal A}}    \nc{\calc}{{\mathcal C}}
\nc{\cald}{\mathcal{D}}     \nc{\cale}{{\mathcal E}}
\nc{\calf}{{\mathcal F}}    \nc{\calg}{{\mathcal G}}
\nc{\calh}{{\mathcal H}}    \nc{\cali}{{\mathcal I}}
\nc{\call}{{\mathcal L}}    \nc{\calm}{{\mathcal M}}
\nc{\caln}{{\mathcal N}}    \nc{\calo}{{\mathcal O}}
\nc{\calp}{{\mathcal P}}    \nc{\calr}{{\mathcal R}}

\nc{\cals}{{\mathcal S}}    \nc{\calt}{{\Omega}}
\nc{\calv}{{\mathcal V}}    \nc{\calu}{{\mathcal U}}
\nc{\calw}{{\mathcal W}}
\nc{\calx}{{\mathcal X}}

\nc{\fraka}{{\mathfrak a}}
\nc{\frakb}{\mathfrak{b}}
\nc{\frakg}{{\frak g}}
\nc{\frakh}{{\frak h}}
\nc{\frakl}{{\frak l}}
\nc{\fraks}{{\frak s}}
\nc{\frakB}{{\frak B}}
\nc{\frakm}{{\frak m}}
\nc{\frakM}{{\frak M}}
\nc{\frakp}{{\frak p}}
\nc{\frakW}{{\frak W}}
\nc{\frakX}{{\frak X}}
\nc{\frakS}{{\frak S}}
\nc{\frakA}{{\frak A}}
\nc{\frakx}{{\frakx}}

\nc{\ynr}[1]{\textcolor{orange}{\underline{Yunnan:}#1 }}

\nc{\lir}[1]{\textcolor{red}{\underline{Li:}#1 }}

\delete{
	\input cyracc.def

	\allowdisplaybreaks
	
	\numberwithin{equation}{section}
	
}

\begin{document}

\title[Based modules over the complex representation ring of $S_4$]{Categorification of based modules over the complex representation ring of $S_4$}

\author{Wenxia Wu}
\address{School of Mathematics and Information Science, Guangzhou University, Waihuan Road West 230, Guangzhou 510006, China}
\email{wxwu@e.gzhu.edu.cn}

\author{Yunnan Li}
\address{School of Mathematics and Information Science, Guangzhou University, Waihuan Road West 230, Guangzhou 510006, China} \email{ynli@gzhu.edu.cn}

\begin{abstract}
The complex representation rings of finite groups are the fundamental class of fusion rings, categorified by the corresponding fusion categories of complex representations. The category of $\mathbb{Z}_+$-modules of finite rank over such a representation ring is also semisimple.

In this paper, we classify the irreducible based modules of rank up to 5 over the complex representation ring $r(S_4)$ of the symmetric group $S_4$.
Totally 16 inequivalent irreducible based modules are obtained.
Based on such a classification result, we further discuss the categorification of  based modules over $r(S_4)$ by module categories over the complex representation category ${\rm Rep}(S_4)$ of $S_4$ arisen from projective representations of certain subgroups of $S_4$.
\end{abstract}

\thanks{2020 {\it Mathematics Subject Classification}. 13C05, 18M20, 19A22.}

\keywords{$\mathbb{Z}_+$-module, $\mathbb{Z}_+$-ring, representation ring, module category, symmetric group}

%\subjclass[2010]{13C05, 18D10, 19A22}

\maketitle

\tableofcontents

\allowdisplaybreaks

\section{Introduction}
\smallskip
Tensor categories should be thought as counterparts of rings in the world of categories, i.e., the categorification of groups and rings ~\cite{BW,DMO,DGNO,ENO1,ENO2,M1}. They are ubiquitous in noncommutative algebra and representation theory. Tensor categories were introduced by B\'{e}nabou \cite{B1} in 1963 and Mac Lane in \cite{M2} as ``categories with multiplication'', and its related theories are now widely used in many fields of mathematics, including algebraic geometry, algebraic topology, number theory, operator algebraic theory, etc. \cite{BJD,ENO3,RSW1,BEK}. The theory of tensor categories is also seen as a sequel developed from that of Hopf algebras and their representation theory \cite{Bi,EKW}. As an important invariant in the theory of tensor categories, the concept of $\mathbb{Z}_+$-ring can be traced back to Lusztig's work~\cite{LuG1} in 1987. Later in \cite{EK,O}, the notion of $\mathbb{Z}_+$-module over $\mathbb{Z}_+$-ring was introduced. Module categories over multitensor categories were first considered in \cite{ENO1,EO}, and then the notion of indecomposable module category was introduced in \cite{O}. As a categorification of irreducible $\mathbb{Z}_+$-modules, it is interesting to classify indecomposable exact module categories over a given tensor category. In this process, it is often necessary to first classify all irreducible $\mathbb{Z}_+$-modules over the Grothendieck ring of a given tensor category.

Typical examples of $\mathbb{Z}_+$-rings are the representation rings of Hopf algebras \cite{CVZ1,HVYZ,LZ1,WZ1,WZ2,W1} and the Grothendieck rings of tensor categories \cite{BS,CF,O1,O2}. It is natural to consider the classification of all irreducible $\mathbb{Z}_+$-modules over them. For example, Etingof and Khovanov classified irreducible $\mathbb{Z}_+$-modules over the group ring $\mathbb ZG$, and showed that indecomposable $\mathbb{Z}_+$-modules over the representation ring of ${\rm SU}(2)$, under certain conditions, correspond to affine and infinite Dynkin diagrams~\cite{EK}. Also, there appears a lot of related research in the context of  near-group fusion categories. For instance, Tambara and Yamagami classified semisimple tensor categories with fusion rules of self-duality for finite abelian groups. Evans, Gannon and Izumi have contributed to the classification of the near-group $C^{*}$-categories \cite{EG,I1}. Li, Yuan and Zhao in \cite{YZL} studied irreducible $\mathbb{Z}_+$-modules of near-group fusion ring $K\left( \mathbb Z_3,3\right) $ and so on.

In this paper, we will explore the problem of classifying irreducible based modules of rank up to 5 over the complex representation ring $r(S_4)$, and then discuss their categorification. In contrast with the representation ring of $S_3$, $r(S_n)$ is no longer a near-group fusion ring when $n>3$, and the classification of irreducible $\mathbb{Z}_+$-modules over general $r(S_n)$ seems to be a hopeless task. In fact, the fusion rule of $r(S_n)$ is already a highly
nontrivial open problem in combinatorics, namely counting the multiplicities of irreducible components of the tensor product of any two
irreducible complex representations of $S_n$ (so called
the Kronecker coefficients).

The paper is organized as follows. In Section~\ref{se:pre}, we recall some basic definitions and propositions. In Section~\ref{se:irred_Z_+}, we discuss the irreducible based modules of rank up to 5  over $r(S_4)$ and give the classification of all these based modules (Propositions~\ref{prop:01}--\ref{prop:02}). In Section~\ref{se:categorified}, we first show that any $\mathbb{Z}_+$-module over the representation ring $r(G)$ of a finite group $G$ categorified by a module category over the representation category ${\rm Rep}(G)$ should be a based module (Theorem~\ref{th:cate_based}),
and then determine which irreducible based modules over $r(S_4)$ can be categorified then (Theorem~\ref{th:classification}).

\section{Preliminaries}\label{se:pre}
  Throughout the paper, all rings are assumed to be associative with unit. Let $\mathbb{Z}_+$ denote the set of nonnegative integers. First we recall the definitions of $\mathbb{Z}_+$-rings and $\mathbb{Z}_+$-modules. For more details about these concepts, readers can refer to~\cite{E,O}.
\subsection{$\mathbb{Z}_+$-rings and $\mathbb{Z}_+$-modules}
\begin{defn}
Let $A$ be a ring which is free as a $\mathbb{Z}$-module.\par
\begin{enumerate}[(i)]
\item A $\mathbb{Z}_+$-{\bf basis} of $A$ is a basis $B=\left \{ b_{i}  \right \} _{i\in I} $ such that $b_{i}b_{j} =\sum_{k\in I} c_{ij}^{k} b_{k}$, where $c_{ij}^{k}\in \mathbb{Z}_+$.
\item A $\mathbb{Z}_+$-{\bf ring} is a ring with a fixed $\mathbb{Z}_+$-basis and with unit 1 being a non-negative linear combination of the basis elments.
\item A $\mathbb{Z}_+$-ring is {\bf unital} if the unit $1$ is one of its basis elements.
    \end{enumerate}
\end{defn}
\begin{defn}\label{def:Z_+-mod}
Let $A$ be a $\mathbb{Z}_+$-ring with basis $\left \{ b_{i} \right \}_{i\in I}$. A $\mathbb{Z}_+$-{\bf module} over $A$ is an $A$-module $M$ with a fixed  $\mathbb{Z}$-basis $\left \{ m_{l}  \right \} _{l\in L}$ such that all the structure constants $a_{il}^{k}$ (defined by the equality $b_{i} m_{l}=\sum_{k} a_{il}^{k} m_{k}$) are non-negative integers.
\end{defn}

A $\mathbb{Z}_+$-module has the following equivalent definition referring to~\cite[Section~3.4]{E}.
\begin{defn}
Let $A$ be a $\mathbb{Z}_+$-ring with basis $\left \{ b_{i} \right \}_{i\in I}$. A $\mathbb{Z}_+$-{\bf module} $M$ over $A$ means an assignment where each basis $b_{i}$ in $A$ is in one-to-one correspondence with a non-negative integer square matrix $M_i$ such that $M$ forms a representation of $A$:\,$M_{i} M_{j}=\sum_{k\in I} c_{ij}^{k} M_{k},\forall \space i,j,k\in I$, and also the unit of $A$ corresponds to the identity matrix. The rank of a $\mathbb{Z}_+$-module $M$ is equal to the order of the matrix $M_{i}$.
\end{defn}

\begin{defn}
(i) Two $\mathbb{Z}_+$-modules $M_{1},\,M_{2}$ over $A$ with bases ${\left \{ m_{i}^{1} \right \}  }_{i\in L_{1}} ,{\left \{ m_{j}^{2} \right \}  }_{j\in L_{2}}$ are {\bf equivalent} if and only if there exists a bijection $\phi :L_{1}\to L_{2}$ such that the induced $\mathbb{Z}$-linear map $ \tilde{\phi} $ of abelian groups $M_{1},\,M_{2}$ defined by $\tilde{\phi} (m_{i}^{1} )=m_{{\phi} (i)}^{2}$ is an isomorphism of $A$-modules. In other words, for $a\in A$, let $a_{M_{1} } $ and $a_{M_{2} } $ be the matrices respect to the bases ${\left \{ m_{i}^{1} \right \}  }_{i\in L_{1}} $ and ${\left \{ m_{j}^{2} \right \}  }_{j\in L_{2}}$, respectively. Then two $\mathbb{Z}_+$-modules $M_{1}$,\,$M_{2}$ of rank $n$ are equivalent if and only if there exists an $n\times n$ permutation matrix $P$ such that $a_{M_{2}}=Pa_{M_{1}}P^{-1}$,$\forall \space a\in A$.
%Basic fact: If an invertible square matrix and its inverse are both non-negative integer matrices, then they are permutation matrices.
\begin{enumerate}
\item[(ii)] The {\bf direct sum} of two $\mathbb{Z}_+$-modules $M_{1},\,M_{2}$ over $A$ is the module $M_{1}\oplus M_{2}$ over $A$ whose basis is the union of the bases of $M_{1}$ and $M_{2}$.
\item[(iii)]A $\mathbb{Z}_+$-module $M$ over $A$ is {\bf indecomposable} if it is not equivalent to a nontrivial direct sum of $\mathbb{Z}_+$-modules.
\item[(iv)]A $\mathbb{Z}_+$-{\bf submodule} of a $\mathbb{Z}_+$-module $M$ over $A$ with basis $\left \{ m_{l} \right \}_{l\in L}$ is a subset $J\subset L$ such that abelian subgroup of $M$ generated by ${\left \{ m_{j} \right \}  }_{j\in J}$ is an $A$-submodule.
\item[(v)]\label{irr-m}A $\mathbb{Z}_+$-module $M$ over $A$ is {\bf irreducible} if any $\mathbb{Z}_+$-submodule of $M$ is 0 or $M$. In other words, the $\mathbb{Z}$-span of any proper subset of the basis of $M$ is not an $A$-submodule.
\end{enumerate}
\end{defn}

\subsection{Based rings and based modules}
\
\newline

\vspace{-.7em}
  Let $A$ be a $\mathbb{Z}_+$-ring with basis $\left \{ b_{i} \right \}_{i\in I}$, and let $I_{0}$ be the set of $i\in I$ such that $b_{i}$ occurs in the decomposition of 1. Let $\tau: A\to \mathbb Z$ denote the group homomorphism defined by
  $$\tau (b_{i} )=\begin{cases}1 & \text{ if } \ i\in I_{0} , \\0  & \text{ if } \ i\notin I_{0} .\end{cases}$$
\begin{defn}
 A $\mathbb{Z}_+$-ring with basis $\left \{ b_{i} \right \}_{i\in I}$ is called a {\bf based ring} if there exists an involution $i \mapsto i^{*}$ of $I$ such that the induced map $$a=\sum_{i\in I} a_{i} b_{i}\mapsto {}a^{*} =\sum_{i\in I} a_{i} b_{i^*},a_{i}\in \mathbb{Z}, $$ is an anti-involution of the ring $A$, and $$\tau (b_{i}b_{j} )=\begin{cases}1 & \text{ if } \ i= j^{*} , \\0  & \text{ if } \  i\ne  j^{*} .\end{cases}$$
 A {\bf fusion ring} is a unital based ring of finite rank.
\end{defn}

\begin{defn}\label{def:based-mod}
 A {\bf based module} over a based ring $A$ with basis $\left \{ b_{i} \right \}_{i\in I}$ is a $\mathbb{Z}_+$-module $M$ with basis $\left \{ m_{l} \right \}_{l\in L}$ over $A$ such that $a_{il}^{k} =a_{i^{\ast } k}^{l} $, where $a_{il}^{k}$ are defined as in Definition~\ref{def:Z_+-mod}.
\end{defn}

Let $A$ be a unital $\mathbb{Z}_+$-ring of finite rank with basis $\left \{ b_{i} \right \}_{i\in I}$ and let $M$ be a $\mathbb{Z}_+$-module over $A$ with $\mathbb{Z}$-basis $\left \{ m_{l}  \right \} _{l\in L}$. Take $b=\sum_{i\in I}b_i$. For any fixed $m_{l_0}$, the $\mathbb{Z}_+$-submodule of $M$ generated by $m_{l_0}$ is the $\mathbb{Z}$-span of $\left \{m_k\right\}_{k\in Y}$, where the set $Y$ consists of $k\in L$ such that $m_k$ is a summand of $bm_{l_0}$. Also, we need the following facts.
\begin{prop}[{\cite[Prop.~3.4.6]{E}}]\label{equ:1}
Let $A$ be a based ring of finite rank over $\mathbb{Z}$. Then there exist only finitely many irreducible $\mathbb{Z}_+$-module over $A$.
\end{prop}

\begin{prop}[{\cite[Exercise~3.4.3.~(i)]{E}}]
A $\mathbb{Z}_+$-module over a based ring is irreducible if and only if it is indecomposable.
\end{prop}
As a result, any $\mathbb{Z}_+$-module of finite rank over a fusion ring is completely reducible, and then only irreducible $\mathbb{Z}_+$-modules need to be classified.

In general, the rank of an irreducible $\mathbb{Z}_+$-module over a fusion ring $A$ may be larger than the rank of $A$; e.g. $A=r(D_5)$ for the dihedral group $D_5$ ({\cite[Remark~1]{JMN}}). In this paper, we explore which irreducible based modules over $r(S_4)$ can be categorified by indecomposable exact module categories over the representation category ${\rm Rep}(S_4)$. Since all these module categories are of rank not greater than 5, we only deal with based modules of rank up to 5 correspondingly.

\section{Irreducible based modules over $r(S_4)$}\label{se:irred_Z_+}
In this section, we will classify the irreducible based modules over the complex representation ring $r(S_4)$ of $S_4$ up to equivalence. $r(S_4)$ is a commutative fusion ring having a $\mathbb{Z}_+$-basis $\left \{ 1,V_{\psi},V_{\rho _{1}},V_{\rho _{2}},V_{\rho _{3}} \right \}$ with the following fusion rule.
\begin{equation}\label{eq:fusion}
\begin{array}{l}
  V_{\psi}^2=1,\quad V_{\psi}V_{\rho _{1}}=V_{\rho _{1}},\quad
  V_{\psi}V_{\rho _{2}}=V_{\rho _{3}},\quad V_{\rho _{1}}^2=1+V_{\psi}+V_{\rho _{1}},\\[.5em]
  V_{\rho_1 }V_{\rho_2 }=V_{\rho_2 }+V_{\rho_3},\quad V_{\rho _{2}}^2=1+ V_{\rho _{1}}+ V_{\rho _{2}}+ V_{\rho _{3}},
\end{array}
\end{equation}
 where $1$, $V_{\psi}$, $V_{\rho _{1}}$ denote the trivial representation, sign representation and 2-dimensional irreducible representation respectively, while $V_{\rho_2}$ stands for the 3-dimensional standard representation and $V_{\rho_3}$ denotes its conjugate representation, then we have
 \smallskip
\begin{table}[h]
    % 插入长度为5pt的垂直空间（也可以是负数，缩进）
    \vspace{-10pt}
    \centering
    % 表名 前面为中文名/后面为英文名
    \captionsetup[table]{position=above}
    % label标签，用以引用本表时。例：autoref{num}
    \label{num}
    % 设置表格单元格的列宽
    \setlength{\tabcolsep}{6mm}{
    % 表示 三线表 有4列
    \scalebox{0.8}{
\begin{tabular}{c|ccccc}
         & $(1)$ & $(12)$ & $(123)$ & $(1234)$ & $(12)(34)$ \\ \hline\\
$\chi_1$ & $1$   &  $1$   &  $1$    & $1$      &  $1$       \\\\
$\chi_\psi$ & $1$   &  $-1$  &  $1$    & $-1$     &  $1$       \\\\
$\chi_{\rho_1}$ & $2$   &  $0$   &  $-1$   & $0$      &  $2$       \\\\
$\chi_{\rho_2}$ & $3$   &  $1$   &  $0$    & $-1$     &  $-1$       \\\\
$\chi_{\rho_3}$ & $3$   &  $-1$  &  $0$    & $1$      &  $-1$       \\
    \end{tabular}}}
     \vspace{.5em}
     \caption{The complex character table of $S_4$}\label{S4}
\end{table}

Let $M$ be a based module of $r(S_4)$ with the basis $\left \{ m_{l}  \right \} _{l\in L}$. Let $T,\,Q,\,U$ and $W$ be the matrices representing the action of $V_{\psi},\,V_{\rho _{1}},\,V_{\rho _{2}},\,V_{\rho _{3}}$ on $M$ respectively. They are all symmetric matrices with nonnegative integer entries by Definition~\ref{def:based-mod}. Let $E$ be the identity matrix. By the fusion rule of $r(S_4)$, %we can exchange $V_{\rho _{2}}$ and $V_{\rho _{3}}$ by multiplying $V_{\psi}$ (i.e. $W=TU$), so
we have
% 多行公式多编号
\begin{align}
\label{eq:01}
&T^2=E,\\
\label{eq:02}
&TQ=QT=Q,\\
\label{eq:03}
&TU=UT=W,\\
\label{eq:04}
&Q^2=E+T+Q,\\
\label{eq:05}
&QU=U+TU,\\
\label{eq:06}
&U^2=E+Q+U+TU.
\end{align}
In particular, since $T^2=E$ and $T$ has nonnegative integer entries, we know that $T$ is a symmetric permutation matrix.

\smallskip\noindent
{\bf Convention.}
Let $P_n$ be the group of $n\times n$ permutation matrices. Since there is naturally a group isomorphism between $S_n$ and $P_n$, we will use the cycle notation of permutations to represent permutation matrices.

\subsection{Irreducible based modules of rank $\le 3$ over $r(S_4)$}
\
\newline

\vspace{-.7em}
We define a $\mathbb{Z}_+$-module $M_{1,1}$ of rank 1 over $r(S_4)$ by letting
\begin{equation}\label{eq:rank1}
V_{\psi} \mapsto 1,\quad V_{\rho _{1}}\mapsto 2,\quad V_{\rho _{2}}\mapsto 3,\quad V_{\rho _{3}}\mapsto 3.
\end{equation}
\begin{prop}\label{prop:01}
Any irreducible based module of rank $1$ over $r(S_4)$ is equivalent to  $M_{1,1}$.
\end{prop}
\begin{proof}
Note that any integral fusion ring $A$ has the unique character ${\rm FPdim}:A\to \mathbb{Z} $ which takes non-negative values on the $\mathbb{Z}_+$-basis, so there exists a unique $\mathbb{Z}_+$-module $M$ of rank 1 over it. Clearly, such $M$ is a based module. Now this argument is available for the situation $A=r(S_4)$.
\end{proof}
Next we consider irreducible based modules of rank $2,3$. According to the fusion rule of $r(S_4)$ given in \eqref{eq:fusion}, it is sufficient to only list the representation matrices of $V_{\psi}$, $V_{\rho _{1}}$ and $V_{\rho _{2}}$ acting on them. For simplicity, we choose to present our result for the cases of small rank 2 and 3 directly, and then analyze the cases of higher rank 4 and 5 with details.
%(Proposition \ref{prop 03}).
\begin{prop}
Let $M$ be an irreducible based module of rank $2$ over $r(S_4)$. Then $M$ is equivalent to one of the based modules $M_{2,i},\,1\leq i\leq 3$, listed in TABLE~\ref{t1}.
 \medskip
\begin{table}[h]
    % 插入长度为5pt的垂直空间（也可以是负数，缩进）
    \vspace{-10pt}
    \centering
    % 表名 前面为中文名/后面为英文名
    \captionsetup[table]{position=above}
    % label标签，用以引用本表时。例：autoref{num}
    \label{num1}
    % 设置表格单元格的列宽
    \setlength{\tabcolsep}{10mm}{
    % 表示 三线表 有4列
    \scalebox{1.0}{
    \begin{tabular}{cccc}
    % toprule表示三线表的顶部线
        \toprule
         &$V_{\psi}$ &$V_{\rho _{1}}$ &$V_{\rho _{2}}$ \\
        % midrule 表示 三线表的 中部线
        \midrule
        % 合并三行1列，用空格代替，也可以用\multirow{}[]{}{} 来表示
        \quad$M_{2,1}$  &$\begin{pmatrix}  1& 0\\  0&1\end{pmatrix}$  &$\begin{pmatrix}  2& 0\\  0&2\end{pmatrix}$   &$\begin{pmatrix}  1& 2\\  2&1\end{pmatrix}$ \\\\
        \quad$M_{2,2}$  &$\begin{pmatrix}  0& 1\\  1&0\end{pmatrix}$  &$\begin{pmatrix}  1& 1\\  1&1\end{pmatrix}$   &$\begin{pmatrix}  2& 1\\  1&2\end{pmatrix}$ \\\\
        \quad$M_{2,3}$  &$\begin{pmatrix}  0& 1\\  1&0\end{pmatrix}$  &$\begin{pmatrix}  1& 1\\  1&1\end{pmatrix}$   &$\begin{pmatrix}  1& 2\\  2&1\end{pmatrix}$ \\
        % bottomrule表示 三线表 的底部线
        \bottomrule
    \end{tabular}}}
    \vspace{.5em}
        \caption{Inequivalent irreducible based modules of rank 2 over $r(S_4)$}\label{t1}
\end{table}
\end{prop}

\begin{prop}
Let $M$ be an irreducible based module of rank $3$ over $r(S_4)$. Then $M$ is equivalent to one of the based modules $M_{3,i},\,1\leq i\leq 3$, listed in TABLE~\ref{t2}.
\medskip
\begin{table}[h]
    % 插入长度为5pt的垂直空间（也可以是负数，缩进）
    \vspace{-10pt}
    \centering
    % 表名 前面为中文名/后面为英文名
    \captionsetup[table]{position=above}
    % label标签，用以引用本表时。例：autoref{num}
    \label{example}
    % 设置表格单元格的列宽
    \setlength{\tabcolsep}{10mm}{
    % 表示 三线表 有4列
    \scalebox{1.0}{
    \begin{tabular}{cccc}
    % toprule表示三线表的顶部线
        \toprule
         &$V_{\psi}$ &$V_{\rho _{1}}$ &$V_{\rho _{2}}$ \\
        % midrule 表示 三线表的 中部线
        \midrule
        % 合并三行1列，用空格代替，也可以用\multirow{}[]{}{} 来表示
        \quad$M_{3,1}$  &$\begin{pmatrix}  1& 0&0\\  0&1&0\\0& 0&1\\\end{pmatrix}$  &$\begin{pmatrix}  0& 1&1\\  1&0&1\\1& 1&0\\\end{pmatrix}$   &$\begin{pmatrix}  1& 1&1\\ 1& 1&1\\1& 1&1\\\end{pmatrix}$ \\[1.5em]
        \quad$M_{3,2}$  &$\begin{pmatrix}  0& 1&0\\  1&0&0\\0& 0&1\\\end{pmatrix}$  &$\begin{pmatrix}  0& 0&1\\ 0& 0&1\\1& 1&1\\\end{pmatrix}$   &$\begin{pmatrix}  1& 0&1\\ 0& 1&1\\1& 1&2\\\end{pmatrix}$ \\[1.5em]
        \quad$M_{3,3}$  &$\begin{pmatrix}  0& 1&0\\  1&0&0\\0& 0&1\\\end{pmatrix}$  &$\begin{pmatrix}  0& 0&1\\ 0& 0&1\\1& 1&1\\\end{pmatrix}$   &$\begin{pmatrix}  0& 1&1\\ 1& 0&1\\1& 1&2\\\end{pmatrix}$ \\[1em]
        % bottomrule表示 三线表 的底部线
        \bottomrule
    \end{tabular}}}
    \vspace{.5em}
        \caption{Inequivalent irreducible based modules of rank 3 over $r(S_4)$}\label{t2}
\end{table}
\end{prop}

\subsection{Irreducible based modules of rank $4,\,5$ over $r(S_4)$}
\begin{prop}\label{prop:03}
Let $M$ be an irreducible based module of rank $4$ over $r(S_4)$. Then $M$ is equivalent to one of the based modules $M_{4,i},\,1\leq i\leq 7$, listed in TABLE~\ref{t3}.
\end{prop}
%\vspace{-3em}
\begin{table}[h]
    % 插入长度为5pt的垂直空间（也可以是负数，缩进）
    %\vspace{-10pt}
    \centering
    % 表名 前面为中文名/后面为英文名
    \captionsetup[table]{position=above}

    % label标签，用以引用本表时。例：autoref{num}
    \label{num}
    % 设置表格单元格的列宽
    \setlength{\tabcolsep}{10mm}{
    % 表示 三线表 有4列
    \scalebox{0.9}{%\renewcommand{\arraystretch}{1.2}
    \begin{tabular}{cccc}
    % toprule表示三线表的顶部线
        \toprule
         &$V_{\psi}$ &$V_{\rho _{1}}$ &$V_{\rho _{2}}$ \\
        % midrule 表示 三线表的 中部线
        \midrule
        % 合并三行1列，用空格代替，也可以用\multirow{}[]{}{} 来表示
        \quad$M_{4,1}$  &$\begin{pmatrix}  1& 0&0&0\\  0&1&0&0\\0&0&1&0\\0&0& 0&1\\\end{pmatrix}$  &$\begin{pmatrix}  0& 0&1&1\\ 0&2&0&0\\1&0&0&1\\1&0&1&0\\\end{pmatrix}$   &$\begin{pmatrix}  0& 1&0&0\\  1&2&1&1\\0&1&0&0\\0&1&0&0\\\end{pmatrix}$ \\[2em]
        \quad$M_{4,2}$  &$\begin{pmatrix}  1& 0&0&0\\  0&1&0&0\\0&0&1&0\\0&0& 0&1\\\end{pmatrix}$  &$\begin{pmatrix}  2& 0&0&0\\  0&2&0&0\\0&0&2&0\\0&0&0&2\\\end{pmatrix}$   &$\begin{pmatrix}  0& 1&1&1\\  1&0&1&1\\1&1&0&1\\1&1&1&0\\\end{pmatrix}$ \\[2em]
        \quad$M_{4,3}$  &$\begin{pmatrix}  0& 1&0&0\\  1&0&0&0\\0&0&1&0\\0&0&0&1\\\end{pmatrix}$  &$\begin{pmatrix}  1& 1&0&0\\  1&1&0&0\\0&0&2&0\\0&0&0&2\\\end{pmatrix}$   &$\begin{pmatrix}  1&0&1&1\\  0&1&1&1\\1&1&0&1\\1&1&1&0\\\end{pmatrix}$ \\[2em]
        \quad$M_{4,4}$  &$\begin{pmatrix}  0& 1&0&0\\  1&0&0&0\\0&0&1&0\\0&0&0&1\\\end{pmatrix}$  &$\begin{pmatrix}  1& 1&0&0\\  1&1&0&0\\0&0&2&0\\0&0&0&2\\\end{pmatrix}$   &$\begin{pmatrix}  0&1&1&1\\  1&0&1&1\\1&1&0&1\\1&1&1&0\\\end{pmatrix}$ \\[2em]
        \quad$M_{4,5}$  &$\begin{pmatrix}  0& 1&0&0\\  1&0&0&0\\0&0&0&1\\0&0& 1&0\\\end{pmatrix}$  &$\begin{pmatrix}  1& 1&0&0\\  1&1&0&0\\0&0&1&1\\0&0&1&1\\\end{pmatrix}$   &$\begin{pmatrix}  0&1&1&1\\  1&0&1&1\\1&1&0&1\\1&1&1&0\\\end{pmatrix}$ \\[2em]
        \quad$M_{4,6}$  &$\begin{pmatrix}  0& 1&0&0\\  1&0&0&0\\0&0&0&1\\0&0& 1&0\\\end{pmatrix}$  &$\begin{pmatrix}  1& 1&0&0\\  1&1&0&0\\0&0&1&1\\0&0&1&1\\\end{pmatrix}$   &$\begin{pmatrix}  0&1&1&1\\  1&0&1&1\\1&1&1&0\\1&1&0&1\\\end{pmatrix}$ \\[2em]
        \quad$M_{4,7}$  &$\begin{pmatrix}  0& 1&0&0\\  1&0&0&0\\0&0&0&1\\0&0& 1&0\\\end{pmatrix}$  &$\begin{pmatrix}  1& 1&0&0\\ 1&1&0&0\\0&0&1&1\\0&0&1&1\\\end{pmatrix}$   &$\begin{pmatrix}  1&0&1&1\\  0&1&1&1\\1&1&1&0\\1&1&0&1\\\end{pmatrix}$ \\[1em]
        % bottomrule表示 三线表 的底部线
        \bottomrule
    \end{tabular}}}
    \vspace{.5em}
    \caption{Inequivalent  irreducible based modules of rank 4 over $r(S_4)$}\label{t3}
\end{table}
\begin{proof}
Let $M$ be a based module of rank 4 over $r(S_4)$, with the action of $r(S_4)$  on it  given by
\begin{equation}
V_{\psi} \mapsto T,\quad V_{\rho _{1}}\mapsto Q=(a_{ij})_{1\leq i,j\leq4} ,\quad V_{\rho _{2}}\mapsto U=(b_{ij})_{1\leq i,j\leq4},\quad V_{\rho _{3}}\mapsto W=TU,\nonumber
\end{equation}
where $a_{ij}=a_{ji}$, $b_{ij}=b_{ji}$.

%We claim that it is enough to deal with one of the following 3 cases for $T$:

The symmetric group $S_4$ has two conjugacy classes of permutations of order 2. One conjugacy class of 6 permutations includes $(12)$, and the other one of 3 permutations includes $(12)(34)$. As previously seen, $T$ is the unit or an element of order 2 in $P_4$, so we have 10 candidates for $T$, and each of them is conjugate to one of the following 3 matrices.
%$T_i\ (i=1,2,3)$.
\begin{equation*}
T_1=E_4,\quad T_2=(12)
%\begin{pmatrix}
%  0&1&0&0 \\
%  1&0&0&0 \\
%  0&0&1&0 \\
%  0&0&0&1
%\end{pmatrix}
,\quad T_3=(12)(34)
%\begin{pmatrix}
%  0&1&0&0 \\
%  1&0&0&0 \\
%  0&0&0&1 \\
%  0&0&1&0
%\end{pmatrix}
.
\end{equation*}
Hence, for the based module $M$ determined by the pair $(T,Q,U)$, there exists a $4\times 4$ permutation matrix $P$ such that $T'=PTP^{-1}$ is one of the above $T_r$'s $(1\leq r\leq 3)$. Correspondingly, let $Q'=PQP^{-1}$, $U'=PUP^{-1}$. Then we get a based module $M'$ determined by the pair $(T',Q',U')$ and equivalent to $M$ as based modules by Definition \ref{irr-m} (i). So we have reduced the proof to the situation when $T=T_r$.

\begin{spacing}{1.5}
\noindent\textbf{Case 1}\quad $T=T_1=E_4$.
\end{spacing}
Since $Q$ satisfies Eq.~\eqref{eq:04}, we obtain the following system of integer equations:
\begin{equation*}
\left\{\begin{array}{l}
a_{11}^2+a_{12}^2+a_{13}^2+a_{14}^2=2+a_{11},\\
a_{11}a_{12}+a_{12}a_{22}+a_{13}a_{23}+a_{14}a_{24}=a_{12},\\
a_{11}a_{13}+a_{12}a_{23}+a_{13}a_{33}+a_{14}a_{34}=a_{13},\\
a_{11}a_{14}+a_{12}a_{24}+a_{13}a_{34}+a_{14}a_{44}=a_{14},\\
a_{12}^2+a_{22}^2+a_{23}^2+a_{24}^2=2+a_{22},\\
a_{12}a_{13}+a_{22}a_{23}+a_{23}a_{33}+a_{24}a_{34}=a_{23},\\
a_{12}a_{14}+a_{22}a_{24}+a_{23}a_{34}+a_{24}a_{44}=a_{24},\\
a_{13}^2+a_{23}^2+a_{33}^2+a_{34}^2=2+a_{33},\\
a_{13}a_{14}+a_{23}a_{24}+a_{33}a_{34}+a_{34}a_{44}=a_{34},\\
a_{14}^2+a_{24}^2+a_{34}^2+a_{44}^2=2+a_{44}.
\end{array}\right.
\end{equation*}
We use Matlab to figure out all the solutions of $Q$ as follows.
$$\begin{array}{l}
Q_1=\begin{pmatrix}
  0&0&1&1 \\
  0&2&0&0 \\
  1&0&0&1 \\
  1&0&1&0
\end{pmatrix},\quad Q_2=\begin{pmatrix}
  0&1&0&1 \\
  1&0&0&1 \\
  0&0&2&0 \\
  1&1&0&0
\end{pmatrix},\quad
Q_3=\begin{pmatrix}
  0&1&1&0 \\
  1&0&1&0 \\
  1&1&0&0 \\
  0&0&0&2
\end{pmatrix},\\[2em]
Q_4=\begin{pmatrix}
  2&0&0&0 \\
  0&0&1&1 \\
  0&1&0&1 \\
  0&1&1&0
\end{pmatrix},\quad Q_5=\begin{pmatrix}
  2&0&0&0 \\
  0&2&0&0 \\
  0&0&2&0 \\
  0&0&0&2
\end{pmatrix}.
\end{array}$$
Next we calculate $U$ after taking $Q$ as one $Q_k$ $\left ( 1\le k\le5 \right ) $.
\begin{spacing}{1.5}
\noindent{\textbf{Case 1.1}\quad $Q=Q_1$}.
\end{spacing}

Since $U$ satisfies Eq.~\eqref{eq:05}, we get
\begin{equation*}
  \begin{split}
  b_{12}&=b_{23}=b_{24},    \\
  b_{11}&=b_{13}=b_{14}=b_{33}=b_{34}=b_{44}.
  \end{split}
  \end{equation*}
Then by Eq.~\eqref{eq:06}, we have
\begin{equation*}
\left\{\begin{array}{l}
 3b_{11}^2+b_{12}^2=2b_{11}+1, \\
 3b_{11}b_{12}+b_{12}b_{22}=2b_{12},\\
 3b_{12}^2+b_{22}^2=2b_{22}+3.
\end{array}\right.
\end{equation*}
The solutions of $U$ given by Matlab are as follows.
\begin{equation*}
U_1=\begin{pmatrix}
  0&1&0&0 \\
  1&2&1&1 \\
  0&1&0&0 \\
  0&1&0&0
\end{pmatrix},\quad U_2=\begin{pmatrix}
  1&0&1&1 \\
  0&3&0&0 \\
  1&0&1&1 \\
  1&0&1&1
\end{pmatrix}.
\end{equation*}
It is easy to check that the based module determined by $\left ( T_1,Q_1,U_1\right )$ is an irreducible based module denoted as $M_{4,1}$, while the based module determined by $\left ( T_1,Q_1,U_2\right )$ is reducible.

Note that there exists a permutation matrix $P=(14)(23)$ such that $PQ_1P^{-1}=Q_2$. Let $U'_1=PU_1P^{-1}$. There is an irreducible based module $N'$ determined by the pair $\big ( T_1,Q_2,U'_1\big)$ and equivalent to $M_{4,1}$ by Definition \ref{irr-m} (i).
Conversely, any irreducible based module with representation matrices $T_1$ and $Q_2$ is equivalent to $M_{4,1}$. The same analysis tells us that irreducible based modules with representation matrices $T_1$ and $Q_3$ (or $Q_4$) are also equivalent to $M_{4,1}$.
\begin{spacing}{1.5}
\noindent{\textbf{Case 1.2}\quad $Q=Q_5$}.
\end{spacing}
\par Since $U$ satisfies Eqs.~\eqref{eq:05} and \eqref{eq:06}, we get a system of integer equations as follows:
\begin{equation*}
\left\{\begin{array}{l}
 b_{11}^2+b_{12}^2+b_{13}^2+b_{14}^2=2b_{11}+3, \\
 b_{11}b_{12}+b_{12}b_{22}+b_{13}b_{23}+b_{14}b_{24}=2b_{12},\\
 b_{11}b_{13}+b_{12}b_{23}+b_{13}b_{33}+b_{14}b_{34}=2b_{13},\\
 b_{11}b_{14}+b_{12}b_{24}+b_{13}b_{34}+b_{14}b_{44}=2b_{14},\\
 b_{12}^2+b_{22}^2+b_{23}^2+b_{24}^2=2b_{22}+3, \\
 b_{12}b_{13}+b_{22}b_{23}+b_{23}b_{33}+b_{24}b_{34}=2b_{23},\\
 b_{12}b_{14}+b_{22}b_{24}+b_{23}b_{34}+b_{24}b_{44}=2b_{24},\\
 b_{13}^2+b_{23}^2+b_{33}^2+b_{34}^2=2b_{33}+3, \\
 b_{13}b_{14}+b_{23}b_{24}+b_{33}b_{34}+b_{34}b_{44}=2b_{34},\\
 b_{14}^2+b_{24}^2+b_{34}^2+b_{44}^2=2b_{44}+3. \\
\end{array}\right.
\end{equation*}
Thus, the solutions of $U$ by Matlab are as follows:
$$\begin{array}{l}
U_1=\begin{pmatrix} 0&1&1&1 \\
  1&0&1&1 \\
  1&1&0&1 \\
  1&1&1&0\end{pmatrix},\quad U_2=\begin{pmatrix} 1&0&0&2 \\
  0&1&2&0 \\
  0&2&1&0 \\
  2&0&0&1\end{pmatrix},\quad U_3=\begin{pmatrix}1&0&2&0 \\
  0&1&0&2 \\
  2&0&1&0 \\
  0&2&0&1\end{pmatrix},\quad
U_4=\begin{pmatrix}1&2&0&0 \\
  2&1&0&0 \\
  0&0&1&2 \\
  0&0&2&1\end{pmatrix},\\[2em]
U_5=\begin{pmatrix} 1&0&2&0 \\
  0&3&0&0 \\
  2&0&1&0 \\
  0&0&0&3\end{pmatrix},\quad U_6=\begin{pmatrix} 1&0&0&2 \\
  0&3&0&0 \\
  0&0&3&0 \\
  2&0&0&1\end{pmatrix},\quad U_7=\begin{pmatrix}1&2&0&0 \\
  2&1&0&0 \\
  0&0&3&0 \\
  0&0&0&3\end{pmatrix},\quad
U_8=\begin{pmatrix}3&0&0&0 \\
  0&3&0&0 \\
  0&0&1&2 \\
  0&0&2&1\end{pmatrix},\\[2em]
U_9=\begin{pmatrix} 3&0&0&0 \\
  0&1&2&0 \\
  0&2&1&0 \\
  0&0&0&3\end{pmatrix},\quad U_{10}=\begin{pmatrix} 3&0&0&0 \\
  0&1&0&2 \\
  0&0&3&0 \\
  0&2&0&1\end{pmatrix},\quad U_{11}=\begin{pmatrix}3&0&0&0 \\
  0&3&0&0 \\
  0&0&3&0 \\
  0&0&0&3\end{pmatrix}.
\end{array}$$
Since $T_1$ and $Q_5$ are diagonal and the solutions $U_t$ $\left ( 2\le t\le11 \right ) $ are block diagonal with at least two blocks, only the based module determined by $(T_1,Q_5,U_1)$ is irreducible, denoted as $M_{4,2}$.

%\ynr{$T_2$, $T_3$ are no longer diagonal, a more complicated situation need to be discussed?}
\begin{spacing}{1.5}
\noindent{\textbf{Case 2}\quad $T=T_2=(12)$}.
\end{spacing}
Since $Q$ satisfies Eq.~\eqref{eq:02}, we get
\begin{equation*}
Q=\begin{pmatrix}
  a_{11}&a_{11}&a_{13}&a_{14} \\
  a_{11}&a_{11}&a_{13}&a_{14}\\
  a_{13}&a_{13}&a_{33}&a_{34} \\
  a_{14}&a_{14}&a_{34}&a_{44}
\end{pmatrix}.
\end{equation*}
Since $Q$ also satisfies Eq.~\eqref{eq:04}, we have the following system of integer equations:
\begin{equation*}
\left\{\begin{array}{l}
 2a_{11}^2+a_{13}^2+a_{14}^2=a_{11}+1, \\
 2a_{11}a_{13}+a_{13}a_{33}+a_{14}a_{34}=a_{13},\\
 2a_{11}a_{14}+a_{13}a_{34}+a_{14}a_{44}=a_{14},\\
 2a_{13}^2+a_{33}^2+a_{34}^2=a_{33}+2,\\
 2a_{13}a_{14}+a_{33}a_{34}+a_{34}a_{44}=a_{34},\\
 2a_{14}^2+a_{34}^2+a_{44}^2=a_{44}+2. \\
\end{array}\right.
\end{equation*}
Hence, the solutions of $Q$ by Matlab are as follows:
$$\begin{array}{l}
Q_1=\begin{pmatrix} 0&0&0&1 \\
  0&0&0&1 \\
  0&0&2&0 \\
  1&1&0&1\end{pmatrix},\quad Q_2=\begin{pmatrix} 0&0&1&0 \\
  0&0&1&0 \\
  1&1&1&0 \\
  0&0&0&2\end{pmatrix},\quad Q_3=\begin{pmatrix}1&1&0&0 \\
  1&1&0&0 \\
  0&0&2&0 \\
  0&0&0&2\end{pmatrix}.
\end{array}$$
Since $U$ satisfies Eq.~\eqref{eq:03}, we get
\begin{equation*}
U=\begin{pmatrix}
  b_{11}&b_{12}&b_{13}&b_{14}\\
  b_{12}&b_{11}&b_{13}&b_{14}\\
  b_{13}&b_{13}&b_{33}&b_{34} \\
  b_{14}&b_{14}&b_{34}&b_{44}
\end{pmatrix}.
\end{equation*}
Next we calculate $U$ after taking $Q$ as one $Q_k$ $\left ( 1\le k\le3 \right )$.
\begin{spacing}{1.5}
\noindent{\textbf{Case 2.1}\quad $Q=Q_1$}.
\end{spacing}

Since $U$ satisfies Eqs.~\eqref{eq:05} and ~\eqref{eq:06}, the solutions of $U$ given by Matlab are as follows.
$$\begin{array}{l}
U_1=\begin{pmatrix} 0&1&0&1 \\
  1&0&0&1 \\
  0&0&3&0 \\
  1&1&0&2\end{pmatrix},\quad U_2=\begin{pmatrix} 1&0&0&1 \\
  0&1&0&1 \\
  0&0&3&0 \\
  1&1&0&2\end{pmatrix}.
\end{array}$$
Since $T_2$, $Q_1$ and all the solutions $U_t$ for $t=1,2$ are block diagonal with at least two blocks, the based modules determined by each pair $(T_2,Q_1,U_t)$ are reducible.

Note that there exists a permutation matrix $P=(12)(34)$ such that $PQ_1P^{-1}=Q_2$. Let $U'_t=PU_tP^{-1}$. Then each based module $N_t$ determined by the pair $\left ( T_2,Q_2,U'_t\right )$ is reducible. Namely, any based module with representation matrices $T_2$ and $Q_2$ is reducible.

\begin{spacing}{1.5}
\noindent{\textbf{Case 2.2}\quad $Q=Q_3$}.
\end{spacing}

Since $U$ satisfies Eqs.~\eqref{eq:05} and \eqref{eq:06}, we have
$$\begin{array}{l}
U_1=\begin{pmatrix} 0&1&1&1 \\
  1&0&1&1 \\
  1&1&0&1 \\
  1&1&1&0\end{pmatrix},\quad U_2=\begin{pmatrix} 1&0&1&1 \\
  0&1&1&1 \\
  1&1&0&1 \\
  1&1&1&0\end{pmatrix},\quad U_3=\begin{pmatrix}1&2&0&0 \\
  2&1&0&0 \\
  0&0&1&2 \\
  0&0&2&1\end{pmatrix},\quad
U_4=\begin{pmatrix}1&2&0&0 \\
  2&1&0&0 \\
  0&0&3&0 \\
  0&0&0&3\end{pmatrix},\\[2em]
U_5=\begin{pmatrix} 2&1&0&0 \\
  1&2&0&0 \\
  0&0&1&2 \\
  0&0&2&1\end{pmatrix},\quad U_6=\begin{pmatrix} 2&1&0&0 \\
  1&2&0&0 \\
  0&0&3&0 \\
  0&0&0&3\end{pmatrix}.
\end{array}$$
Since $T_2$, $Q_3$ and the solutions $U_s$ $(3\le s\le 6)$ are block diagonal with at least two blocks, only the based module determined by $(T_2,Q_3,U_1)$ and $(T_2,Q_3,U_2)$ are irreducible, denoted as $M_{4,3}$ and $M_{4,4}$ respectively. It is easy to check that $M_{4,3}$ and $M_{4,4}$ are inequivalent based modules.

\begin{spacing}{1.5}
\noindent{\textbf{Case 3}\quad $T=T_3=(12)(34)$}.
\end{spacing}
Since $Q$ satisfies Eq.~\eqref{eq:02}, we get
\begin{equation*}
Q=\begin{pmatrix}
  a_{11}&a_{11}&a_{13}&a_{13} \\
  a_{11}&a_{11}&a_{13}&a_{13}\\
  a_{13}&a_{13}&a_{33}&a_{33} \\
  a_{13}&a_{13}&a_{33}&a_{33}
\end{pmatrix}.
\end{equation*}
Then by Eq.~\eqref{eq:04}, we have the following system of integer equations:
\begin{equation*}
\left\{\begin{array}{l}
 2a_{11}^2+2a_{13}^2=a_{11}+1, \\
 2a_{11}a_{13}+2a_{13}a_{33}=a_{13},\\
 2a_{13}^2+2a_{33}^2=a_{33}+1.
\end{array}\right.
\end{equation*}
$Q$ has the following unique solution.
\begin{equation*}
Q_1=\begin{pmatrix}
  1&1&0&0\\
  1&1&0&0\\
  0&0&1&1\\
  0&0&1&1
\end{pmatrix}.
\end{equation*}

Since $U$ satisfies Eq.~\eqref{eq:03}, we get
\begin{equation*}
U=\begin{pmatrix}
  b_{11}&b_{12}&b_{13}&b_{14}\\
  b_{12}&b_{11}&b_{14}&b_{13}\\
  b_{13}&b_{14}&b_{33}&b_{34} \\
  b_{14}&b_{13}&b_{34}&b_{33}
\end{pmatrix}.
\end{equation*}
Since $U$ also satisfies Eqs.~\eqref{eq:05} and \eqref{eq:06}, we obtain the solutions of $U$ by Matlab as follows:
$$\begin{array}{l}
U_1=\begin{pmatrix} 0&1&1&1 \\
  1&0&1&1 \\
  1&1&0&1 \\
  1&1&1&0\end{pmatrix},\quad U_2=\begin{pmatrix} 0&1&1&1 \\
  1&0&1&1 \\
  1&1&1&0 \\
  1&1&0&1\end{pmatrix},\quad U_3=\begin{pmatrix}1&0&1&1 \\
  0&1&1&1 \\
  1&1&1&0 \\
  1&1&0&1\end{pmatrix},\quad U_4=\begin{pmatrix}1&0&1&1\\
  0&1&1&1 \\
  1&1&0&1 \\
  1&1&1&0\end{pmatrix},\\[2em]
U_5=\begin{pmatrix} 1&2&0&0 \\
  2&1&0&0 \\
  0&0&1&2 \\
  0&0&2&1\end{pmatrix},\quad U_6=\begin{pmatrix} 1&2&0&0 \\
  2&1&0&0 \\
  0&0&2&1 \\
  0&0&1&2\end{pmatrix},\quad U_7=\begin{pmatrix}2&1&0&0 \\
  1&2&0&0 \\
  0&0&1&2 \\
  0&0&2&1\end{pmatrix},\quad
U_8=\begin{pmatrix}2&1&0&0 \\
  1&2&0&0 \\
  0&0&2&1 \\
  0&0&1&2\end{pmatrix}.
\end{array}$$
Clearly, $T_3$, $Q_1$ and the solutions $U_s$ are block diagonal with at least two blocks, but the based module determined by the pair $\left ( T_3,Q_1,U_t\right )$ is irreducible, denoted as $M_{4,s}$, where $5\le s \le8, 1\le t\le4$. Define the $\mathbb Z$-module isomorphism $\phi :M_{4,6}\to M_{4,8}$ by
\begin{equation*}
\phi(v_1^1)=v_4^2,\quad \phi(v_2^1)=v_3^2,\quad \phi(v_3^1)=v_2^2,\quad\phi(v_4^1)=v_1^2.
\end{equation*}
It is easy to see that $M_{4,6}$ is equivalent to $M_{4,8}$ as based modules over $r(S_4)$ under $\phi$. Then we can check that $\left \{ M_{4,s} \right \} _{5\le s \le7} $ are inequivalent irreducible based modules.
\end{proof}

Finally, we construct two based modules $M_{5,i}\ \left (i=1,2 \right )$ over $r(S_4)$ %with basis $\left \{ v_1^i,v_2^i,v_3^i,v_4^i,v_5^i\right \} $ and
with the actions of $r(S_4)$ on them presented in TABLE~\ref{t4}.

%\vspace{-3em}
\begin{table}[h]
    % 插入长度为5pt的垂直空间（也可以是负数，缩进）
    %\vspace{-10pt}
    \centering
    % 表名 前面为中文名/后面为英文名
    \captionsetup[table]{position=above}

    % label标签，用以引用本表时。例：autoref{num}
    \label{num}
    % 设置表格单元格的列宽
    \setlength{\tabcolsep}{10mm}{
    % 表示 三线表 有4列
    \scalebox{0.9}{%\renewcommand{\arraystretch}{1.2}
    \begin{tabular}{cccc}
    % toprule表示三线表的顶部线
        \toprule
         &$V_{\psi}$ &$V_{\rho _{1}}$ &$V_{\rho _{2}}$ \\
        % midrule 表示 三线表的 中部线
        \midrule
        % 合并三行1列，用空格代替，也可以用\multirow{}[]{}{} 来表示
        \quad$M_{5,1}$  &$\begin{pmatrix}  0& 1&0&0&0\\  1&0&0&0&0\\0&0&0&1&0\\0&0& 1&0&0\\0&0& 0&0&1\\\end{pmatrix}$  &$\begin{pmatrix}  0& 0&0&0&1\\  0& 0&0&0&1\\0&0&1&1&0\\0&0& 1&1&0\\1&1& 0&0&1\\\end{pmatrix}$   &$\begin{pmatrix}  0& 0&0&1&0\\  0&0&1&0&0\\0&1&1&1&1\\1&0& 1&1&1\\0&0& 1&1&0\\\end{pmatrix}$ \\[2.5em]
        \quad$M_{5,2}$  &$\begin{pmatrix}  0& 1&0&0&0\\  1&0&0&0&0\\0&0&0&1&0\\0&0& 1&0&0\\0&0& 0&0&1\\\end{pmatrix}$  &$\begin{pmatrix}  1& 1&0&0&0\\  1& 1&0&0&0\\0&0&1&1&0\\0&0& 1&1&0\\0&0& 0&0&2\\\end{pmatrix}$   &$\begin{pmatrix}  0& 0&0&1&1\\  0& 0&1&0&1\\0&1&0&0&1\\1&0& 0&0&1\\1&1& 1&1&1\\\end{pmatrix}$ \\[2em]
        % bottomrule表示 三线表 的底部线
        \bottomrule
    \end{tabular}}}
    \vspace{.5em}
    \caption{Inequivalent  irreducible based modules of rank 5 over $r(S_4)$}\label{t4}
\end{table}
\begin{prop}\label{prop:02}
Let $M$ be an irreducible based module of rank $5$ over $r(S_4)$. Then $M$ is equivalent to one of the based modules $M_{5,i}\,(i=1,2)$, listed in TABLE~\ref{t4}.
\end{prop}
\begin{proof}
Let $M$ be a based module of rank 5 over $r(S_4)$, with the action of $r(S_4)$ on it given by
\begin{equation*}
V_{\psi} \mapsto T,\quad V_{\rho _{1}}\mapsto Q=(a_{ij})_{1\leq i,j\leq5} ,\quad V_{\rho _{2}}\mapsto U=(b_{ij})_{1\leq i,j\leq5},\quad V_{\rho _{3}}\mapsto W=TU,
\end{equation*}
where $a_{ij}=a_{ji}$, $b_{ij}=b_{ji}$.

First by the similar argument applied in the case of rank 4, we only need to deal with one of the following 3 cases for $T$.
\begin{equation*}
T_1=E_5,\quad T_2=(12)
%\begin{pmatrix}
%  0&1&0&0&0 \\
%  1&0&0&0&0 \\
%  0&0&1&0&0 \\
%  0&0&0&1&0 \\
%  0&0&0&0&1
%\end{pmatrix}
,\quad T_3=(12)(34)
%\begin{pmatrix}
%  0&1&0&0&0 \\
%  1&0&0&0&0 \\
%  0&0&0&1&0 \\
%  0&0&1&0&0 \\
%  0&0&0&0&1
%\end{pmatrix}
.
\end{equation*}
\begin{spacing}{1.5}
\noindent{\textbf{Case 1}\quad $T=T_1=E_5$}.
\end{spacing}
There are 11 solutions of $Q$ satisfying Eq.~\eqref{eq:04},
but only two conjugacy classes by permutation matrices with their representatives given as follows.
\begin{equation*}
Q_1=\begin{pmatrix} 0&1&1&0&0 \\
  1&0&1&0&0 \\
  1&1&0&0&0 \\
  0&0&0&2&0 \\
  0&0&0&0&2
  \end{pmatrix},\quad Q_2=\begin{pmatrix} 2&0&0&0&0 \\
  0&2&0&0&0 \\
  0&0&2&0&0 \\
  0&0&0&2&0 \\
  0&0&0&0&2
  \end{pmatrix}.
\end{equation*}
Next we calculate $U$ after taking $Q$ as one $Q_k$ $\left ( k=1,2 \right )$.
\begin{spacing}{1.5}
\noindent{\textbf{Case 1.1}\quad $Q=Q_1$}.
\end{spacing}

There are 4 solutions of $U$ satisfying Eqs.~\eqref{eq:05} and \eqref{eq:06} as follows.
\begin{equation*}
U_1=\begin{pmatrix}
  0&0&0&1&0 \\
  0&0&0&1&0 \\
  0&0&0&1&0 \\
  1&1&1&2&0 \\
  0&0&0&0&3
  \end{pmatrix},\ \ U_2=\begin{pmatrix}
  0&0&0&0&1 \\
  0&0&0&0&1 \\
  0&0&0&0&1 \\
  0&0&0&3&0 \\
  1&1&1&0&2
  \end{pmatrix},\ \  U_3=\begin{pmatrix}
  1&1&1&0&0 \\
  1&1&1&0&0 \\
  1&1&1&0&0 \\
  0&0&0&3&0 \\
  0&0&0&0&3 \\
  \end{pmatrix},\ \ U_4=\begin{pmatrix}
  1&1&1&0&0 \\
  1&1&1&0&0 \\
  1&1&1&0&0 \\
  0&0&0&1&2 \\
  0&0&0&2&1
  \end{pmatrix}.
\end{equation*}
\begin{spacing}{1.5}
\noindent{\textbf{Case 1.2}\quad $Q=Q_2$}.
\end{spacing}

There are 31 solutions of $U$ satisfying Eqs.~\eqref{eq:05} and \eqref{eq:06},
but only 4 conjugacy classes by permutation matrices and their representatives as follows.
\begin{equation*}
U_1=\begin{pmatrix}
  0&1&1&1&0 \\
  1&0&1&1&0 \\
  1&1&0&1&0 \\
  1&1&1&0&0 \\
  0&0&0&0&3
  \end{pmatrix},\ \ U_2=\begin{pmatrix}
  1&0&0&2&0 \\
  0&1&2&0&0 \\
  0&2&1&0&0 \\
  2&0&0&1&0 \\
  0&0&0&0&3
  \end{pmatrix},\ \  U_3=\begin{pmatrix}
  1&2&0&0&0 \\
  2&1&0&0&0 \\
  0&0&3&0&0 \\
  0&0&0&3&0 \\
  0&0&0&0&3 \\
  \end{pmatrix},\ \  U_4=\begin{pmatrix}
  3&0&0&0&0 \\
  0&3&0&0&0 \\
  0&0&3&0&0 \\
  0&0&0&3&0 \\
  0&0&0&0&3
  \end{pmatrix}.
\end{equation*}
Each pair $(T_1,Q_k,U_r)$ above determines a based module, but not irreducible for $1\le r\le4$.
\begin{spacing}{1.5}
\noindent{\textbf{Case 2}\quad $T=T_2=(12)$}.
\end{spacing}

There are 5 solutions of $Q$ satisfying Eq.~\eqref{eq:04},
but only 3 conjugacy classes with the following representatives.
\begin{equation*}
Q_1=\begin{pmatrix} 0&0&0&0&1 \\
  0&0&0&0&1 \\
  0&0&2&0&0 \\
  0&0&0&2&0 \\
  1&1&0&0&1
  \end{pmatrix},\quad Q_2=\begin{pmatrix} 1&1&0&0&0 \\
  1&1&0&0&0 \\
  0&0&0&1&1 \\
  0&0&1&0&1 \\
  0&0&1&1&0
  \end{pmatrix},\quad Q_3=\begin{pmatrix} 1&1&0&0&0 \\
  1&1&0&0&0 \\
  0&0&2&0&0 \\
  0&0&0&2&0 \\
  0&0&0&0&2
  \end{pmatrix}.
\end{equation*}
Next we calculate $U$ after choosing $Q$.
\begin{spacing}{1.5}
\noindent{\textbf{Case 2.1}\quad $Q=Q_1$}.
\end{spacing}

There are 4 solutions of $U$ satisfying Eqs.~\eqref{eq:05} and \eqref{eq:06} as follows.
\begin{equation*}
U_1=\begin{pmatrix}
  0&1&0&0&1 \\
  1&0&0&0&1 \\
  0&0&1&2&0 \\
  0&0&2&1&0 \\
  1&1&0&0&2
  \end{pmatrix},\ \ U_2=\begin{pmatrix}
  0&1&0&0&1 \\
  1&0&0&0&1 \\
  0&0&3&0&0 \\
  0&0&0&3&0 \\
  1&1&0&0&2
  \end{pmatrix},\ \ U_3=\begin{pmatrix}
  1&0&0&0&1 \\
  0&1&0&0&1 \\
  0&0&1&2&0 \\
  0&0&2&1&0 \\
  1&1&0&0&2 \\
  \end{pmatrix},\ \ U_4=\begin{pmatrix}
  1&0&0&0&1 \\
  0&1&0&0&1 \\
  0&0&3&0&0 \\
  0&0&0&3&0 \\
  1&1&0&0&2
  \end{pmatrix}.
\end{equation*}

\begin{spacing}{1.5}
\noindent{\textbf{Case 2.2}\quad $Q=Q_2$}.
\end{spacing}

There are 2 solutions of $U$ satisfying Eqs.~\eqref{eq:05} and \eqref{eq:06} as follows.
\begin{equation*}
U_1=\begin{pmatrix}
  1&2&0&0&0 \\
  2&1&0&0&0 \\
  0&0&1&1&1 \\
  0&0&1&1&1 \\
  0&0&1&1&1
  \end{pmatrix},\ \ U_2=\begin{pmatrix}
  2&1&0&0&0 \\
  1&2&0&0&0 \\
  0&0&1&1&1 \\
  0&0&1&1&1 \\
  0&0&1&1&1
  \end{pmatrix}.
\end{equation*}

\begin{spacing}{1.5}
\noindent{\textbf{Case 2.3}\quad $Q=Q_3$}.
\end{spacing}
There are 14 solutions of $U$ satisfying Eqs.~\eqref{eq:05} and \eqref{eq:06},
but only 6 conjugacy classes by permutation matrices and their representatives as follows.
$$\begin{array}{l}
U_1=\begin{pmatrix}
  0&1&1&1&0 \\
  1&0&1&1&0 \\
  1&1&0&1&0 \\
  1&1&1&0&0 \\
  0&0&0&0&3
  \end{pmatrix},\quad U_2=\begin{pmatrix}
  1&0&1&1&0 \\
  0&1&1&1&0 \\
  1&1&0&1&0 \\
  1&1&1&0&0 \\
  0&0&0&0&3
  \end{pmatrix},\quad U_3=\begin{pmatrix}
  1&2&0&0&0 \\
  2&1&0&0&0 \\
  0&0&1&2&0 \\
  0&0&2&1&0 \\
  0&0&0&0&3
  \end{pmatrix},\\[2.5em]
U_4=\begin{pmatrix}
  2&1&0&0&0 \\
  1&2&0&0&0 \\
  0&0&1&2&0 \\
  0&0&2&1&0 \\
  0&0&0&0&3
  \end{pmatrix},\quad
U_5=\begin{pmatrix}
  1&2&0&0&0 \\
  2&1&0&0&0 \\
  0&0&3&0&0 \\
  0&0&0&3&0 \\
  0&0&0&0&3 \\
  \end{pmatrix},\quad
U_6=\begin{pmatrix}
  2&1&0&0&0 \\
  1&2&0&0&0  \\
  0&0&3&0&0 \\
  0&0&0&3&0 \\
  0&0&0&0&3
  \end{pmatrix}.
\end{array}$$
Through analysis, all based modules derived from Case 2 are reducible.

\begin{spacing}{1.5}
\noindent{\textbf{Case 3}\quad $T=T_3=(12)(34)$}.
\end{spacing}
There are 3 solutions of $Q$ satisfying Eq.~\eqref{eq:04} as follows.
\begin{equation*}
Q_1=\begin{pmatrix} 0&0&0&0&1 \\
  0&0&0&0&1 \\
  0&0&1&1&0 \\
  0&0&1&1&0 \\
  1&1&0&0&1
  \end{pmatrix},\quad Q_2=\begin{pmatrix} 1&1&0&0&0 \\
  1&1&0&0&0 \\
  0&0&1&1&0 \\
  0&0&1&1&0 \\
  0&0&0&0&2
  \end{pmatrix},\quad Q_3=\begin{pmatrix} 1&1&0&0&0 \\
  1&1&0&0&0 \\
  0&0&0&0&1 \\
  0&0&0&0&1 \\
  0&0&1&1&1  \end{pmatrix}.
\end{equation*}
Next we calculate $U$ after fixing $Q$.
\begin{spacing}{1.5}
\noindent{\textbf{Case 3.1}\quad $Q=Q_1$}.
\end{spacing}

There are 6 solutions of $U$ satisfying Eqs.~\eqref{eq:05} and \eqref{eq:06} as follows.
$$\begin{array}{l}
U_1=\begin{pmatrix}
  0&0&0&1&0 \\
  0&0&1&0&0 \\
  0&1&1&1&1 \\
  1&0&1&1&1 \\
  0&0&1&1&0
  \end{pmatrix},\quad U_2=\begin{pmatrix}
  0&0&1&0&0 \\
  0&0&0&1&0 \\
  1&0&1&1&1 \\
  0&1&1&1&1 \\
  0&0&1&1&0
  \end{pmatrix},\quad U_3=\begin{pmatrix}
  0&1&0&0&1 \\
  1&0&0&0&1 \\
  0&0&1&2&0 \\
  0&0&2&1&0 \\
  1&1&0&0&2
  \end{pmatrix},\\[2.5em]
U_4=\begin{pmatrix}
  0&1&0&0&1 \\
  1&0&0&0&1 \\
  0&0&2&1&0 \\
  0&0&1&2&0 \\
  1&1&0&0&2
  \end{pmatrix},\quad
U_5=\begin{pmatrix}
  1&0&0&0&1 \\
  0&1&0&0&1 \\
  0&0&1&2&0 \\
  0&0&2&1&0 \\
  1&1&0&0&2 \\
  \end{pmatrix},\quad
U_6=\begin{pmatrix}
  1&0&0&0&1 \\
  0&1&0&0&1 \\
  0&0&2&1&0 \\
  0&0&1&2&0 \\
  1&1&0&0&2
  \end{pmatrix}.
\end{array}$$
Each pair $(T_3,Q_1,U_r)$ $\left ( 1\le r\le6 \right ) $ above determines a based module, but only the based modules with representation matrices $U_1$ and $U_2$ are irreducible. Denote $M_{5,1}$ and $M'_{5,1}$ these two irreducible based modules with the corresponding $\mathbb Z$-basis $\left \{ v_1^k,v_2^k,v_3^k,v_4^k,v_5^k\right \} $ for $k=1,2$ respectively. Define the $\mathbb Z$-module isomorphism $\phi :M_{5,1}\to M'_{5,1}$ by
\begin{equation*}
\phi(v_s^1)=v_s^2,\quad \phi(v_3^1)=v_4^2,\quad \phi(v_4^1)=v_3^2,\quad s=1,2,5.
\end{equation*}
Then it is easy to see that $M_{5,1}$ is equivalent to $M'_{5,1}$ as based modules over $r(S_4)$ under $\phi$.

\begin{spacing}{1.5}
\noindent{\textbf{Case 3.2}\quad $Q=Q_2$}.
\end{spacing}

There are 10 solutions of $U$ satisfying Eqs.~\eqref{eq:05} and \eqref{eq:06},
but only 7 conjugacy classes with their representatives given as follows.
$$\begin{array}{l}
U_1=\begin{pmatrix}
  0&0&0&1&1 \\
  0&0&1&0&1 \\
  0&1&0&0&1 \\
  1&0&0&0&1 \\
  1&1&1&1&1
  \end{pmatrix},\ \  U_2=\begin{pmatrix}
  0&1&1&1&0 \\
  1&0&1&1&0 \\
  1&1&0&1&0 \\
  1&1&1&0&0 \\
  0&0&0&0&3 \\
  \end{pmatrix},\ \ U_3=\begin{pmatrix}
  0&1&1&1&0 \\
  1&0&1&1&0 \\
  1&1&1&0&0 \\
  1&1&0&1&0 \\
  0&0&0&0&3 \\
  \end{pmatrix},\ \
U_4=\begin{pmatrix}
  1&0&1&1&0 \\
  0&1&1&1&0 \\
  1&1&1&0&0 \\
  1&1&0&1&0 \\
  0&0&0&0&3 \\
  \end{pmatrix},\\[2.5em]
  U_5=\begin{pmatrix}
  1&2&0&0&0 \\
  2&1&0&0&0 \\
  0&0&1&2&0 \\
  0&0&2&1&0 \\
  0&0&0&0&3
  \end{pmatrix},\ \  U_6=\begin{pmatrix}
  1&2&0&0&0 \\
  2&1&0&0&0 \\
  0&0&2&1&0 \\
  0&0&1&2&0 \\
  0&0&0&0&3
  \end{pmatrix},\ \  U_7=\begin{pmatrix}
  2&1&0&0&0 \\
  1&2&0&0&0 \\
  0&0&2&1&0 \\
  0&0&1&2&0 \\
  0&0&0&0&3
  \end{pmatrix}.
\end{array}$$
Each pair $(T_3,Q_2,U_t)$ $\left ( 2\le t\le7 \right ) $ above determines a based module, but only the based module with representation matrix $U_1$ is irreducible. We denote it by $M_{5,2}$.

Also, the based modules obtained by taking $Q=Q_3$ are equivalent to the based module $M_{5,1}$ found in Case 3.1.
\end{proof}

\section{Categorified based modules by module categories over ${\rm Rep}(S_4)$}%\label{sec:module_cat}
\label{se:categorified}

%\indent
In this section, we will apply the knowledge of module categories over the complex representation category of a finite group to find which based modules over $r(S_4)$ can be categorified by module categories over the representation category ${\rm Rep}(S_4)$ of $S_4$. For the details about module categories over tensor categories, see e.g. \cite[Chapter 7]{E}.

\smallskip%\medskip
First we recall the required result for the upcoming discussion. For any finite group $G$, the second cohomology group $H^2(G,\CC^*)$ is known to be a finite abelian group called the {\bf Schur multiplier} and classifies central extensions of $G$. The notion of universal central extension of a finite group was first investigated by Schur in \cite{Sc}.

Let ${\rm Rep}(G,\alpha)$ denote the semisimple abelian category of projective representations of $G$ with the multiplier $\alpha\in Z^2(G,\CC^*)$.
Equivalently, ${\rm Rep}(G,\alpha)$ is the representation category of the twisted group algebra $\CC G_\alpha$ of $G$ with multiplication
$$g\cdot_\alpha h=\alpha(g,h)gh,\quad g,h\in G.$$
In particular, ${\rm Rep}(G,\alpha)={\rm Rep}(G)$ when taking $\alpha=1$.

Let $\alpha \in Z^2(G,\CC^*)$ represent an element of order $d$ in $H^2(G,\CC^*)$.
Define
$${\rm Rep}^\alpha(G)=\bigoplus _{j=0}^{d-1}{\rm Rep}(G,\alpha^j).$$
According to the result in~\cite{C}, we know that ${\rm Rep}^\alpha(G)$ becomes a fusion category with the tensor product of two projective representations in ${\rm Rep}(G,\alpha^i)$ and ${\rm Rep}(G,\alpha^j)$ respectively lying in ${\rm Rep}(G,\alpha^{i+j})$, and the dual object in ${\rm Rep}(G,\alpha^i)$ lying in ${\rm Rep}(G,\alpha^{d-i})$. Correspondingly, we have the fusion ring
\begin{equation}\label{eq:proj_rep_ring}
r^\alpha(G)=\bigoplus _{j=0}^{d-1}r(G,\alpha^j).
\end{equation}

Now let $H$ be a subgroup of $G$ and $\alpha \in Z^2(H,\CC^*)$. The category ${\rm Rep}(H,\alpha)$ is a module category over ${\rm Rep}(G)$ by applying the restriction functor ${\rm Res}_H^{G}:{\rm Rep}(G)\to {\rm Rep}(H)$.
\begin{theorem}[{\cite[Theorem 3.2]{O}}]
The indecomposable exact module categories over the representation category ${\rm Rep}(G)$ are of the form ${\rm Rep}(H,\alpha)$ and are classified by conjugacy classes of pairs $(H,[\alpha])$.
\end{theorem}
Consequently by \cite[Prop.~7.7.2]{E}, we know that
\begin{prop}
The Grothendieck group
$$r(H,\alpha)={\rm Gr}({\rm Rep}(H,\alpha))$$
is an irreducible $\mathbb{Z}_+$-module over $r(G)$.
\end{prop}

%\yn{First briefly recall the definition of module category over a fusion category,
%then state \cite[Proposition 4.1]{EO} clearly!}

\medskip

Next we show that any $\mathbb{Z}_+$-module over the complex representation ring $r(G)$ of a finite group $G$ categorified by this way is a based module.
\begin{theorem}\label{th:cate_based}
Let $G$ be a finite group, $H$ a subgroup of $G$ and $\alpha \in  Z^2(H,\CC^*)$. The $\mathbb{Z}_+$-module $r(H,\alpha)$ over $r(G)$ is a based module.
\end{theorem}
\begin{proof}
Let $\left \{ {\psi}_i \right \} _{i\in I}$ be the $\mathbb{Z}_+$-basis of $r(G)$.
Take $r^\alpha(H)$ defined in Eq.~\eqref{eq:proj_rep_ring} as a $\mathbb{Z}_+$-module over $r(G)$ with the $\mathbb{Z}$-basis $\left \{ {\chi}_k \right \} _{k\in J}$ such that
$$\psi_i.\chi_k=\sum_l a_{ik}^l \chi_l,\quad a_{ik}^l\in \mathbb{Z}_+.$$

On the other hand, we write the fusion rule of the fusion ring $r^\alpha(H)$ as follows,
\begin{equation*}
\chi_i\chi_j=\sum_{k=1}^{s} n_{ij}^k\chi_k,\quad n_{ij}^k\in \mathbb{Z}_+.
\end{equation*}
Since the number $n_{ij}^{k^*}$ is invariant under cyclic permutations of $i,j,k$% and $i\mapsto i^*$ induces an anti-involution of the fusion ring $r^\alpha(H)$
, we have
$$n_{ij}^k=n_{k^*i}^{j^*}=n_{i^*k}^{j}.$$
By the restriction rule, we interpret $r(G)$ as a subring of $r^\alpha(H)$ and write down
$$\psi_i=\sum_j r_{ij} \chi_j,\quad r_{ij}\in \mathbb{Z}_+.$$
Then
$$\psi_i.\chi_k=\sum_j r_{ij} \chi_j\chi_k= \sum_{j,\,l} r_{ij}n_{jk}^l \chi_l.$$
By comparing the coefficients, we see that
$$a_{ik}^l=\sum_{j} r_{ij}n_{jk}^l
=\sum_{j} r_{ij}n_{j^*l}^k=\sum_{j} r_{i^*j^*}n_{j^*l}^k=\sum_{j} r_{i^*j}n_{jl}^k=a_{i^*l}^k,
$$
so $r^\alpha(H)$ is a based module over $r(G)$, and $r(H,\alpha)$ is clearly a based submodule of $r^\alpha(H)$. Equivalently, any $\mathbb{Z}_+$-module over $r(G)$ categorified by a module category ${\rm Rep}(H,\alpha)$  over ${\rm Rep}(G)$ must be a based module.
\end{proof}

\smallskip

By Theorem~\ref{th:cate_based}, we only need to focus on those inequivalent irreducible based modules $M_{i,j}$ over $r(S_4)$ collected in Section~\ref{se:irred_Z_+}
%TABLE~\ref{t5}
, each of which is possibly categorified by a module category ${\rm Rep}(H,\alpha)$ for some $H<S_4$ and $\alpha \in Z^2(H,\CC^*)$.

All the non-isomorphic subgroups of the symmetric group $S_4$ are as follows:
\begin{enumerate}[(i)]
\item\label{g1}
the symmetric group $S_3$;
\item\label{g2}
the cyclic groups $\mathbb Z_i,\,1\leq i\leq 4$;
\item\label{g3}
the Klein 4-group $K_4$;
\item\label{g4}
the alternating group $A_4$;
\item\label{g5}
the dihedral group $D_4$;
\item\label{g6}
the symmetric group $S_4$ itself.
\end{enumerate}

Correspondingly, the Schur multipliers we concern here are given as follows (see e.g. \cite{Ka}),
$$H^2(\mathbb Z_n,\CC^*)\cong H^2(S_3,\CC^*)\cong 0,\ n\geq 1,\quad \ H^2(K_4,\CC^*)\cong H^2(D_4,\CC^*)\cong H^2(A_4,\CC^*)\cong H^2(S_4,\CC^*)\cong\mathbb Z_2.$$
As a result, we only need to consider the following two situations.

(1) Module category ${\rm Rep}(H)$ for any subgroup $H<S_4$;

(2) Module category ${\rm Rep}(H,\alpha)$ for any subgroup $H<S_4$ and nontrivial twist $\alpha \in Z^2(H,\CC^*)$.

\subsection{The module categories over ${\rm Rep}(S_4)$ with trivial twists}
\
\newline

\vspace{-.7em}
\ref{g1} First we consider the representation category ${\rm Rep}(S_3)$ as a module category over ${\rm Rep}(S_4)$.
\begin{theorem}
$r(S_3) = {\rm Gr}({\rm Rep}(S_3))$ is an irreducible based module over $r(S_4) = {\rm Gr}({\rm Rep}(S_4))$ equivalent to the based module $M_{3,2}$ in TABLE~\ref{t2}.
\end{theorem}
\begin{proof}
According to the branching rule of symmetric groups (see e.g. \cite[Theorem 2.8.3]{Sa}),
we have the following restriction rules.
\begin{align*}
{\rm Res}_{S_3}^{S_4} \left ( 1 \right )=1,\quad {\rm Res}_{S_3}^{S_4} \left ( V_{\psi} \right )=\chi,\quad {\rm Res}_{S_3}^{S_4} \left ( V_{\rho_{1}} \right )=V,\\
{\rm Res}_{S_3}^{S_4} \left ( V_{\rho_{2}} \right )=1+V,\quad {\rm Res}_{S_3}^{S_4} \left ( V_{\rho_{3}} \right )=\chi+V,
\end{align*}
where $\chi$ and $V$ denote the sign representation and the standard representation in ${\rm Rep}({S_3})$ respectively.
Hence we get the representation matrices of basis elements of $r(S_4)$ acting on $r({S_3})$ as follows.
\begin{align*}
1 \mapsto E_3,\quad V_{\psi} \mapsto \begin{pmatrix}0& 1 &0 \\1&0  &0 \\0&0  &1\end{pmatrix},\quad V_{\rho _{1}}\mapsto \begin{pmatrix}0& 0 &1 \\0& 0 &1  \\1& 1&1 \end{pmatrix},\quad V_{\rho _{2}}\mapsto \begin{pmatrix}1& 0 &1 \\0& 1 &1  \\1& 1&2 \end{pmatrix},\quad V_{\rho _{3}}\mapsto \begin{pmatrix}0& 1 &1 \\1& 0 &1  \\1& 1&2 \end{pmatrix}.
\end{align*}
We see that $r(S_3)$ is an irreducible based module $M_{3,2}$  according to TABLE~\ref{t2}. In other words, the based module $M_{3,2}$ can be categorified by the module category ${\rm Rep}({S_3})$ over ${\rm Rep}(S_4)$.
\end{proof}
\begin{remark}\label{rk:symm}
Since the roles of the standard representation and its dual in $r(S_4)$ are symmetric, we can exchange the notations $V_{\rho_{2}}$ and $V_{\rho_{3}}$ for them to get the following restriction rules instead.
\begin{equation*}
{\rm Res}_{S_3}^{S_4} \left ( V_{\rho_{2}} \right )=\chi+V,\quad {\rm Res}_{S_3}^{S_4} \left ( V_{\rho_{3}} \right )=1+V.
\end{equation*}
Therefore, we get another action of $r(S_4)$ on $r({S_3})$ such that $r(S_3)$ is an irreducible based module over $r(S_4)$ equivalent to the based module $M_{3,3}$ according to TABLE~\ref{t2}. In other words, the based module $M_{3,3}$ can also be categorified by the module category ${\rm Rep}({S_3})$ over ${\rm Rep}(S_4)$.
\end{remark}

\ref{g2} Secondly, we consider ${\rm Rep}(\mathbb{Z}_4)$ as a module category over ${\rm Rep}(S_4)$.
\begin{theorem}
$r(\mathbb{Z}_4)={\rm Gr}({\rm Rep}(\mathbb{Z}_4))$ is an irreducible based module over $r(S_4)$ equivalent to the based module $M_{4,5}$ in TABLE~\ref{t3}.
\end{theorem}
\begin{proof}
Let $\mathbb{Z}_4=\left \{ 1,g,g^2,g^3 \right \} $ be the cyclic group of order $4$, and it has four non-isomorphic 1-dimensional irreducible representations, denoted by $U_i,\,i=0,1,2,3$. Let $U_0=1$ represent  the trivial representation and
\begin{equation*}
\chi_{U_1}(g) =\sqrt{-1} ,\quad \chi_{U_2}(g) =-1,\quad\chi_{U_3}(g) =-\sqrt{-1}.
\end{equation*}
On the other hand, we consider $\mathbb{Z}_4$ as the subgroup of $S_4$ generated by $g=(1234)$. Then by the character table of $S_4$ (TABLE~\ref{S4}), we have
\begin{align*}
&\chi_\psi(g^i)=(-1)^i,\quad \chi_{\rho_1}(g^i)=1+(-1)^i,\quad
\chi_{\rho_2}(g^i)=(-1)^i+(\sqrt{-1})^i+(-\sqrt{-1})^i,\\
&\chi_{\rho_3}(g^i)=1+(\sqrt{-1})^i+(-\sqrt{-1})^i.
\end{align*}
So the restriction rule of $r(S_4)$ on $r(\mathbb{Z}_4)$ is given as follows.
\begin{align*}
&{\rm Res}_{\mathbb{Z}_4}^{S_4} \left ( 1 \right )=1,\quad {\rm Res}_{\mathbb{Z}_4}^{S_4} \left ( V_{\psi} \right )=U_2,\quad {\rm Res}_{\mathbb{Z}_4}^{S_4} \left ( V_{\rho_{1}} \right )=1+U_2,\\
&{\rm Res}_{\mathbb{Z}_4}^{S_4} \left ( V_{\rho_{2}} \right )=U_1+U_2+U_3,\quad{\rm Res}_{\mathbb{Z}_4}^{S_4} \left ( V_{\rho_{3}} \right )=1+U_1+U_3.
\end{align*}
Then we get the representation matrices of basis elements of $r(S_4)$ acting on $r(\mathbb{Z}_4)$ as follows.
\begin{align*}
&1 \mapsto E_4,\quad V_{\psi} \mapsto \begin{pmatrix}0& 0 &1 &0\\0& 0 &0&1\\1& 0&0&0\\0&1&0&0\end{pmatrix},\quad V_{\rho _{1}}\mapsto \begin{pmatrix}1&0&1&0\\0&1&0&1\\1&0&1&0\\0&1&0&1\end{pmatrix},\\
&V_{\rho _{2}}\mapsto \begin{pmatrix}0&1&1&1\\1&0&1&1\\1&1&0&1\\1&1&1&0 \end{pmatrix},\quad V_{\rho_{3}}\mapsto \begin{pmatrix}1&1&0&1\\1&1&1&0\\0&1&1&1\\1&0&1&1\end{pmatrix}.
\end{align*}
Let $\left \{ w_i\right \}_{1\leq i\leq 4} $  be the stated $\mathbb{Z}$-basis of $M_{4,5}$, and define a $\mathbb{Z}$-linear map $\varphi:M_{4,5}\to r(\mathbb{Z}_4)$ by
\begin{equation*}
\varphi  (w_1)=U_3,\quad \varphi  (w_2)=U_1,\quad \varphi  (w_3)=U_2,\quad\varphi  (w_4)=1.
\end{equation*}
Then it is easy to check that $\varphi $ is an isomorphism of $r(S_{4})$-modules, so $M_{4,5}\cong r(\mathbb{Z}_4)$ as based modules by Definition \ref{irr-m} (i). In other words, based module $M_{4,5}$ can be categorified by the module category ${\rm Rep}(\mathbb{Z}_4)$ over ${\rm Rep}(S_4)$.
\end{proof}
\begin{remark}By the same argument as in Remark~\ref{rk:symm}, $V_{\rho _{2}}$ and $V_{\rho _{3}}$ can be required to satisfy the following restriction rules instead.
\begin{equation*}
{\rm Res}_{\mathbb{Z}_4}^{S_4} \left ( V_{\rho_{2}} \right )=1+U_1+U_3,\quad{\rm Res}_{\mathbb{Z}_4}^{S_4} \left ( V_{\rho_{3}} \right )=U_1+U_2+U_3.
\end{equation*}
Therefore, we get another action of $r(S_4)$ on $r({\mathbb{Z}_4})$ such that $r(\mathbb{Z}_4)$ is an irreducible based module over $r(S_4)$ equivalent to the based module $M_{4,7}$ according to TABLE~\ref{t3}. In other words, the based module $M_{4,7}$ can also be categorified by the module category ${\rm Rep}(\mathbb{Z}_4)$ over ${\rm Rep}(S_4)$.

Also, one can similarly check that the module category ${\rm Rep}(\mathbb{Z}_2)$ over ${\rm Rep}(S_4)$ categorifies the based modules $M_{2,2}$ and $M_{2,3}$, while ${\rm Rep}(\mathbb{Z}_3)$ over ${\rm Rep}(S_4)$ categorifies the based module $M_{3,1}$.
%Since the cohomology group  $H^2(\mathbb Z_n,\CC^*)$ is trivial for an arbitrary cyclic group $\mathbb Z_n$, we know that the remaining based module $M_{2,1}$ can not be categorified.
\end{remark}
\ref{g3} Now we consider ${\rm Rep}(K_4)$ as a module category over ${\rm Rep}(S_4)$.
\begin{theorem}
$r(K_4)={\rm Gr}({\rm Rep}(K_4))$ is an irreducible based module over $r(S_4)$ equivalent to the based module $M_{4,7}$ in TABLE~\ref{t3}.
\end{theorem}
\begin{proof}
We consider $K_4$ as the subgroup of $S_4$ generated by $(12)$ and $(34)$, and it has four non-isomorphic 1-dimensional irreducible representations $Y_0=1$ and $Y_1,Y_2,Y_3$ such that
\begin{align*}
&\chi_{Y_1}((12)) =-1,\quad \chi_{Y_1}((34)) =1;\quad\chi_{Y_2}((12)) =1,\quad\chi_{Y_2}((34)) =-1;\\
&\chi_{Y_3}((12))=-1,\quad \chi_{Y_3}((34)) = -1.
\end{align*}
On the other hand, by the character table of $S_4$ (TABLE~\ref{S4}), we have
\begin{align*}
&\chi_\psi((12))=\chi_\psi((34))=-1,\quad \chi_\psi((12)(34))=1;\\
&\chi_{\rho_1}((12))=\chi_{\rho_1}((34))=0,\quad \chi_{\rho_1}((12)(34))=2;\\
&\chi_{\rho_2}((12))=\chi_{\rho_2}((34))=1,\quad \chi_{\rho_2}((12)(34))=-1;\\
&\chi_{\rho_3}((12))=\chi_{\rho_3}((34))=-1,\quad \chi_{\rho_3}((12)(34))=-1.
\end{align*}
So we have the following restriction rules.
\begin{align*}
&{\rm Res}_{K_4}^{S_4} \left ( 1 \right )= 1, \quad {\rm Res}_{K_4}^{S_4} \left ( V_{\psi} \right )=Y_3, \quad {\rm Res}_{K_4}^{S_4} \left ( V_{\rho_{1}} \right )=1+Y_3,\\
&{\rm Res}_{K_4}^{S_4} \left ( V_{\rho_{2}} \right )= 1+Y_1+Y_2,\quad {\rm Res}_{K_4}^{S_4} \left ( V_{\rho_{3}} \right )= Y_1+Y_2+Y_3.
\end{align*}
Then, we get the representation matrices of basis elements of $r(S_4)$ acting on $r(K_4)$ as follows.
\begin{align*}
&1 \mapsto E_4,\quad V_{\psi} \mapsto \begin{pmatrix}0& 0 &0 &1\\0& 0 &1&0  \\0&1&0&0\\1&0&0&0\end{pmatrix},\quad V_{\rho _{1}}\mapsto \begin{pmatrix}1&0&0&1\\0&1&1&0\\0&1&1&0\\1&0&0&1 \end{pmatrix},\\
&V_{\rho _{2}}\mapsto \begin{pmatrix}1&1&1&0\\1&1&0&1\\1&0&1&1\\0&1&1&1\end{pmatrix},\quad V_{\rho_{3}}\mapsto \begin{pmatrix}0&1&1&1\\1&0&1&1\\1&1&0&1\\1&1&1&0  \end{pmatrix}.
\end{align*}
Let $\left \{ w_i\right \}_{1\leq i\leq 4} $  be the stated $\mathbb{Z}$-basis of $M_{4,7}$  listed in TABLE~\ref{t3}, then
\begin{equation*}
w_1\mapsto Y_2,\quad  w_2\mapsto Y_1,\quad  w_3\mapsto 1,\quad w_4\mapsto Y_3,
\end{equation*}
 defines an equivalence of $\mathbb{Z}_+$-modules between $M_{4,7}$ and $r(K_4)$.
%Then $r(K_4)$ is an irreducible based module over $r(S_4)$ equivalent to $M_{4,7}$ listed in TABLE~\ref{t3}.
In other words, the irreducible based module $M_{4,7}$ can be categorified by the module category ${\rm Rep}(K_4)$ over ${\rm Rep}(S_4)$.
 \end{proof}
\begin{remark}
 In a manner analogous to the argument in Remark~\ref{rk:symm}, it follows that the irreducible based module $M_{4,5}$ can also be categorified by the module category ${\rm Rep}(K_4)$ over ${\rm Rep}(S_4)$.
\end{remark}

\medskip
\ref{g4} We consider ${\rm Rep}(A_4)$ as a module category over ${\rm Rep}(S_4)$.
\begin{theorem}\label{theo 02}
$r(A_4)={\rm Gr}({\rm Rep}(A_4))$ is an irreducible based module over $r(S_4)$ equivalent to the based module $M_{4,1}$ in TABLE~\ref{t3}.
\end{theorem}
\begin{proof}
We know that $A_4$ has three non-isomorphic 1-dimensional irreducible representations and one 3-dimensional irreducible representation, denoted by $N_0, N_1, N_2$ and $N_3$ respectively, where $N_0=1$ represents the trivial representation, and
\begin{align*}
&\chi_{N_1}((123)) =\omega ,\quad \chi_{N_1}((12)(34)) =1;\quad \chi_{N_2}((123)) =\omega^2,\quad \chi_{N_2}((12)(34)) =1;\\
&\chi_{N_3}((123)) =\chi_{N_3}((132)) =0,\quad \chi_{N_3}((12)(34)) =-1,\quad \omega =\frac{-1+\sqrt{-3} }{2}.
\end{align*}
On the other hand, the character table of $S_4$ (TABLE~\ref{S4}) tells us that
\begin{align*}
&\chi_\psi((123))=1,\quad \chi_\psi((12)(34))=1;\quad \chi_{\rho_1}((123))=-1,\quad \chi_{\rho_1}((12)(34))=2;\\
&\chi_{\rho_2}((123))=0, \quad \chi_{\rho_2}((12)(34))=-1;\quad \chi_{\rho_3}((123))=0, \quad \chi_{\rho_3}((12)(34))=-1.
\end{align*}
So we have the following restriction rules.
\begin{align*}
{\rm Res}_{A_4}^{S_4} \left ( 1 \right )= {\rm Res}_{A_4}^{S_4} \left ( V_{\psi} \right )=1,\quad {\rm Res}_{A_4}^{S_4} \left ( V_{\rho_{1}} \right )=N_1+N_2,\quad {\rm Res}_{A_4}^{S_4} \left ( V_{\rho_{2}} \right )={\rm Res}_{A_4}^{S_4} \left ( V_{\rho_{3}} \right )=N_3.
\end{align*}
Hence, we get the representation matrices of basis elements of $r(S_4)$ acting on $r(A_4)$ as follows.
\begin{align*}
1 \mapsto E_4,\quad V_{\psi} \mapsto E_4,\quad V_{\rho _{1}}\mapsto \begin{pmatrix}0& 1 &1 &0\\1& 0 &1&0  \\1& 1&0&0\\0&0&0&2  \end{pmatrix},\quad V_{\rho _{2}},V_{\rho_{3}}\mapsto \begin{pmatrix}0& 0 &0 &1\\0& 0 &0 &1\\0& 0 &0 &1\\1&1&1&2  \end{pmatrix}.
\end{align*}
Then $r(A_4)$ is an irreducible based module over $r(S_4)$ equivalent to $M_{4,1}$ listed in TABLE~\ref{t3}. In other words, the irreducible based module $M_{4,1}$ can be categorified by the module category ${\rm Rep}(A_4)$ over ${\rm Rep}(S_4)$.
\end{proof}

\ref{g5} Next we consider ${\rm Rep}(D_4)$ as a module category over ${\rm Rep}(S_4)$.
\begin{theorem}\label{theo 01}
$r(D_4)={\rm Gr}({\rm Rep}(D_4))$ is an irreducible based module over $r(S_4)$ equivalent to the based module $M_{5,2}$ in TABLE~\ref{t4}.
\end{theorem}
\begin{proof}
The dihedral group $D_4=\langle r,s\,|\, r^4=s^2=(rs)^2=1\rangle$ has four 1-dimensional irreducible representations and one 2-dimensional irreducible representation up to isomorphism, denoted by $W_0,W_1,W_2,W_3$ and $W_4$ respectively. Let $W_0=1$ stand for the trivial representation and
\begin{align*}
&\chi_{W_1}(r) =1 ,\quad \chi_{W_1}(s) =-1; \quad\chi_{W_2}(r) =-1,\quad \chi_{W_2}(s) =1;\\
&\chi_{W_3}(r) =-1,\quad\chi_{W_3}(s) =-1; \quad \chi_{W_4}(r) =\chi_{W_4}(s) =\chi_{W_4}(rs) = 0.
\end{align*}
On the other hand, we consider $D_4$ as the subgroup of $S_4$ by taking $r=(1234)$ and $s=(12)(34)$. Then $rs=(13)$. By the character table of $S_4$ (TABLE~\ref{S4}), we have
\begin{align*}
&\chi_\psi((1234))=-1,\quad \chi_\psi((12)(34))=1,\quad \chi_\psi((13))=-1;\\ &\chi_{\rho_1}((1234))=0,\quad \chi_{\rho_1}((12)(34))=2,\quad  \chi_{\rho_1}((13))=0;\\
&\chi_{\rho_2}((1234))=-1, \quad \chi_{\rho_2}((12)(34))=-1, \quad \chi_{\rho_2}((13))=1;\\
&\chi_{\rho_3}((1234))=1, \quad \chi_{\rho_3}((12)(34))=-1, \quad \chi_{\rho_3}((13))=-1.
\end{align*}
So we have the following restriction rules.
\begin{align*}
&{\rm Res}_{D_4}^{S_4} \left ( 1 \right )= 1, \quad {\rm Res}_{D_4}^{S_4} \left ( V_{\psi} \right )=W_2, \quad {\rm Res}_{D_4}^{S_4} \left ( V_{\rho_{1}} \right )=1+W_2,\\
&{\rm Res}_{D_4}^{S_4} \left ( V_{\rho_{2}} \right )= W_3+W_4,\quad {\rm Res}_{D_4}^{S_4} \left ( V_{\rho_{3}} \right )= W_1+W_4.
\end{align*}
Then, we get the representation matrices of basis elements of $r(S_4)$ acting on $r(D_4)$ as follows.
\begin{align*}
&1 \mapsto E_5,\quad V_{\psi} \mapsto \begin{pmatrix}0& 0 &1 &0&0\\0& 0 &0&1&0  \\1& 0&0&0&0\\0&1&0&0&0\\0&0&0&0&1  \end{pmatrix},\quad V_{\rho _{1}}\mapsto \begin{pmatrix}1&0&1&0&0\\0&1&0&1&0\\1&0&1&0&0\\0&1&0&1&0\\0&0&0&0&2  \end{pmatrix},\\
&V_{\rho _{2}}\mapsto \begin{pmatrix}0&0&0&1&1\\0&0&1&0&1\\0&1&0&0&1\\1&0&0&0&1\\1&1&1&1&1  \end{pmatrix},\quad V_{\rho_{3}}\mapsto \begin{pmatrix}0&1&0&0&1\\1&0&0&0&1\\0&0&0&1&1\\0&0&1&0&1\\1&1&1&1&1  \end{pmatrix}.
\end{align*}
Then $r(D_4)$ is an irreducible based module over $r(S_4)$ equivalent to $M_{5,2}$ listed in TABLE~\ref{t4}. In other words, the irreducible based module $M_{5,2}$ can be categorified by the module category ${\rm Rep}(D_4)$ over ${\rm Rep}(S_4)$.
\end{proof}

\ref{g6}
Finally, we consider ${\rm Rep}(S_4)$ as a module category over itself.
\begin{theorem}\label{theo:02}
The regular $\mathbb{Z}_+$-module $r(S_4)$ over itself is equivalent to the irreducible based module $M_{5,1}$ in TABLE~\ref{t4}.
\end{theorem}
\begin{proof}
Let $r(S_4)$ be the regular $\mathbb{Z}_+$-module over itself with the $\mathbb{Z}$-basis $\left \{ 1,V_{\psi},V_{\rho_{1}},V_{\rho_{2}},V_{\rho_{3}} \right \} $, and the action of $r(S_4)$ on it is given as follows.
\begin{align*}
&1 \mapsto E_5,\quad V_{\psi} \mapsto \begin{pmatrix}
  0&1&0&0&0 \\
  1&0&0&0&0 \\
  0&0&1&0&0 \\
  0&0&0&0&1 \\
  0&0&0&1&0
\end{pmatrix},\quad V_{\rho _{1}}\mapsto \begin{pmatrix}
  0&0&1&0&0 \\
  0&0&1&0&0 \\
  1&1&1&0&0 \\
  0&0&0&1&1 \\
  0&0&0&1&1
\end{pmatrix},\\
&V_{\rho _{2}}\mapsto \begin{pmatrix}
  0&0&0&1&0 \\
  0&0&0&0&1 \\
  0&0&0&1&1 \\
  1&0&1&1&1 \\
  0&1&1&1&1
\end{pmatrix},\quad V_{\rho _{3}}\mapsto \begin{pmatrix}
  0&0&0&0&1 \\
  0&0&0&1&0 \\
  0&0&0&1&1 \\
  0&1&1&1&1 \\
  1&0&1&1&1
\end{pmatrix}.
\end{align*}
Then the regular $\mathbb{Z}_+$-module $r(S_4)$ over itself is equivalent to $M_{5,1}$ listed in TABLE~\ref{t4}. In other words, the irreducible based module $M_{5,1}$ over $r(S_4)$ can be categorified by the module category ${\rm Rep}(S_4)$ over itself.
\end{proof}

\begin{remark}
Following the argument presented in Remark~\ref{rk:symm}, if we exchange the notations $V_{\rho _{2}}$ and $V_{\rho _{3}}$ with their restriction rules given in the proof of Theorem~\ref{theo 01} and Theorem~\ref{theo:02}, we see that $r(D_4)$ and $r(S_4)$ are still equivalent to $M_{5,2}$ and $M_{5,1}$ respectively.
\end{remark}

\subsection{The module categories over ${\rm Rep}(S_4)$ with nontrivial twists}
\
\newline

\vspace{-.7em}
Lastly, we consider the module category ${\rm Rep}(H,\alpha)$ over ${\rm Rep}(S_4)$, where $H$ is a subgroup of $S_4$ with $\alpha$ representing the unique nontrivial cohomological class in $H^2(H,\CC^*)$. All non-isomorphic irreducible projective representations of $H$  with the multiplier $\alpha$ form a $\mathbb Z$-basis of $r(H,\alpha)$, whose cardinality is the number of $\alpha$-regular conjugacy classes by \cite[Theorem 6.1.1]{Ka0}.

Firstly, we consider the twisted group algebra of $K_4$. There is only one irreducible projective representation with respect to $\alpha$ up to isomorphism, see e.g. \cite[Appendix D.1]{Ro}. Hence, $r(K_4,\alpha)$ is a based module of rank $1$ over $r(S_4)$ equivalent to $M_{1,1}$ defined in \eqref{eq:rank1}. Namely, the based module $M_{1,1}$ can also be categorified by ${\rm Rep}(K_4,\alpha)$.

Secondly, we consider the twisted group algebra of $D_4$.
\begin{theorem}
$r(D_4,\alpha)={\rm Gr}({\rm Rep}(D_4,\alpha))$ is an irreducible based module over $r(S_4)$ equivalent to the based module $M_{2,3}$ in TABLE~\ref{t1}.
\end{theorem}
\begin{proof}
Let $D_4=\langle r,s\,|\, r^4=s^2=(rs)^2=1\rangle $. Let $\alpha\in  Z^2(D_4,\CC^*)$ be the 2-cocycle defined by
\begin{equation}
 \alpha (r^is^j,r^{i'}s^{j'})=(\sqrt{-1})^{ji'}.\nonumber
\end{equation}
Here $i,i'\in {\left \{ 0,1,2,3 \right \} },\, j,j'\in\left \{ {0,1} \right \} $
% and $\varepsilon^2=-1$
. As shown in \cite[Section 3.7]{Ka1}, this is a unitary 2-cocycle representing the unique non-trivial cohomological class in $H^2(D_4,\CC^*)$. According to \cite[Section 3]{C}, there exist two (2-dimensional) non-isomorphic irreducible projective representations of $D_4$ with respect to $\alpha$, which are given by
\begin{equation}
\begin{aligned}
\pi_{l}: D_4 & \rightarrow \mathrm{GL}_{2}(\mathbb{C}) \\
r^{i} s^{j} & \mapsto A_{l}^{i} B^{j},
\end{aligned}\nonumber
\end{equation}
where $A_l=\begin{pmatrix}(\sqrt{-1})^l &0 \\ 0&(\sqrt{-1})^{1-l}\end{pmatrix} ,B=\begin{pmatrix}0 &1 \\1 &0\end{pmatrix}$, $l=1,2$. Also, for irreducible representations $W_0,W_1,W_2,W_3$ and $W_4$ of $D_4$ mentioned in the proof of Theorem~\ref{theo 01}, we have
% the following tensor product rule in $r^\alpha(D_4)$.
\begin{equation}
W_0\otimes \pi_l=W_1\otimes\pi_l=\pi_l,\quad W_2\otimes \pi_l=W_3\otimes \pi_l=\pi_{3-l},\quad W_4\otimes \pi_l=\pi_1+\pi_2.\nonumber
\end{equation}

Next, using the previous restriction rule of $r(S_4)$ on $r(D_4)$, we get the representation matrices of basis elements of $r(S_4)$ acting on $r(D_4,\alpha)$ as follows.
\begin{equation}
\begin{aligned}
1 \mapsto E_2,\quad V_{\psi} \mapsto \begin{pmatrix}
 0 &1 \\
 1&0
\end{pmatrix},\quad V_{\rho_1} \mapsto \begin{pmatrix}
 1 &1 \\
 1&1
\end{pmatrix},\quad V_{\rho_2} \mapsto \begin{pmatrix}
 1&2 \\
 2&1
\end{pmatrix},\quad V_{\rho_3} \mapsto \begin{pmatrix}
 2&1 \\
 1&2
\end{pmatrix}.\nonumber
\end{aligned}
\end{equation}
Then $r(D_4,\alpha)$ is an irreducible based module over $r(S_4)$ equivalent to $M_{2,3}$ listed in TABLE~\ref{t1}. In other words, the irreducible based module $M_{2,3}$ can be categorified by the module category ${\rm Rep}(D_4,\alpha)$ over ${\rm Rep}(S_4)$.
\end{proof}

\begin{remark}
As discussed in Remark~\ref{rk:symm}, it follows that the irreducible based module $M_{2,2}$ can also be categorified by the module category ${\rm Rep}(D_4,\alpha)$ over ${\rm Rep}(S_4)$.
\end{remark}

Next we consider the twisted group algebras of $A_4$ and $S_4$. By \cite[Theorem 6.1.1]{Ka0}, $A_4$ has three (2-dimensional) non-isomorphic irreducible projective representations, denoted as $V_{\gamma _{1}}, V_{\gamma _{2}}$ and $V_{\gamma _{3}}$ respectively. Similarly, $S_4$ has two (2-dimensional) non-isomorphic irreducible projective representations $V_{\xi_1}, V_{\xi_2}$ and one (4-dimensional) irreducible projective representation $V_{\xi_3}$. We give the character table for projective representations of $A_4$ and $S_4$ as follows, where primes are used to differentiate between the two classes splitting from a single conjugacy class of $A_4$ in its double cover $\tilde A_4$, and the same applies to $S_4$; subscripts distinguish between the two classes splitting from the conjugacy classes $(31)^{'}$ and $(31)^{''}$ in the double cover $\tilde S_4$ of $S_4$ respectively. For more details, see \cite[Chapter 4]{Ho}.
\
\newline
\par
\begin{table}[h]
    % 插入长度为5pt的垂直空间（也可以是负数，缩进）
    \vspace{-10pt}
    \centering
    % 表名 前面为中文名/后面为英文名
    \captionsetup[table]{position=above}
    % label标签，用以引用本表时。例：autoref{num}
    \label{num}
    % 设置表格单元格的列宽
    \setlength{\tabcolsep}{6mm}{
    % 表示 三线表 有4列
    \scalebox{0.8}{
\begin{tabular}{c|ccccccc}
             & $(1^4)^{'}$ & $(1^4)^{''}$ &  $(2^2)$ & $(31)^{'}_1$ & $(31)^{''}_1$& $(31)^{'}_2$& $(31)^{''}_2$\\ \hline\\
$\chi_{\gamma _{1}}$ & $2$   &      $-2$   &        $0$   & $1$      &      $-1$    &      $1$&         $-1$       \\\\
$\chi_{\gamma _{2}}$ & $2$   &      $-2$   &        $0$   & $\omega$ &      $-\omega$     &$\omega^2$&  $-\omega^2$      \\\\
$\chi_{\gamma _{3}}$ & $2$   &      $-2$   &        $0$   & $\omega^2$  &   $-\omega^2$    &$\omega$&  $-\omega$  \\
    \end{tabular}}}
     \vspace{.5em}
     \caption{The character table for irreducible projective representations of $A_4$}\label{t7}
\end{table}
\begin{table}[h]
    % 插入长度为5pt的垂直空间（也可以是负数，缩进）
    \vspace{-10pt}
    \centering
    % 表名 前面为中文名/后面为英文名
    \captionsetup[table]{position=above}
    % label标签，用以引用本表时。例：autoref{num}
    \label{num}
    % 设置表格单元格的列宽
    \setlength{\tabcolsep}{6mm}{
    % 表示 三线表 有4列
    \scalebox{0.8}{
\begin{tabular}{c|cccccccc}
         & $(1^4)^{'}$ & $(1^4)^{''}$ & $(2 1^2)$& $(2^2)$ & $(3 1)^{'}$ & $(31)^{''}$ & $(4)^{'}$& $(4)^{''}$\\ \hline\\
$\chi_{\xi_1}$ & $2$   &  $-2$   &  $0$    & $0$      &  $1$    &  $-1$&  $\sqrt{2} $&  $-\sqrt{2} $       \\\\
$\chi_{\xi_2}$ & $2$   &  $-2$   &  $0$    & $0$     &  $1$    &  $-1$&  $-\sqrt{2} $&  $\sqrt{2} $       \\\\
$\chi_{\xi_3}$ & $4$   &  $-4$   &  $0$   & $0$      &  $-1$    &  $1$&  $0$&  $0$       \\
    \end{tabular}}}
     \vspace{.5em}
\caption{The character table for irreducible projective representations of $S_4$}\label{t8}
\end{table}
In TABLE~\ref{t7} we denote $\omega=e^{2\pi\sqrt{-1}/3} =\frac{-1+\sqrt{-3}}{2}$.  Then we have the following theorems.

\begin{theorem}
$r(A_4,\alpha)={\rm Gr}({\rm Rep}(A_4,\alpha))$ is an irreducible based module over $r(S_4)$ equivalent to the based module $M_{3,1}$ in TABLE~\ref{t2}.
\end{theorem}
\begin{proof}
%Let $\alpha\in Z^2(A_4,\CC^*)$ be a cocycle.
For the irreducible representations $N_0,N_1,N_2$ and $N_3$ of $A_4$ mentioned in the proof of Theorem~\ref{theo 02}, we obtain the following tensor product rule in $r^\alpha(A_4)$ by computing the values of products of characters,
\begin{align*}
&N_0 \otimes V_{\gamma_i }= V_{\gamma_{i} },\quad N_1 \otimes V_{\gamma_j }= V_{\gamma_{j+1} },\quad N_1 \otimes V_{\gamma_3 }= V_{\gamma_{1} };\\
&N_2 \otimes V_{\gamma_1 }= V_{\gamma_{3} },\quad N_2 \otimes V_{\gamma_2 }= V_{\gamma_{1} },\quad N_2 \otimes V_{\gamma_3 }= V_{\gamma_{2} };\\
&N_3 \otimes V_{\gamma_i }= V_{\gamma_{1} }+V_{\gamma_{2} }+V_{\gamma_{3} },
\end{align*}
where $i=1,2,3$, $j=1,2$. Next, by combining this with the previous restriction rule of $r(S_4)$ on $r(A_4)$, we obtain
$$
1, V_{\psi} \mapsto E_3,\quad V_{\rho_1} \mapsto \begin{pmatrix}
 0 &1 &1\\
 1 &0 &1\\
 1 &1 &0
\end{pmatrix},\quad V_{\rho_2},V_{\rho_3} \mapsto \begin{pmatrix}
 1 &1 &1\\
 1 &1 &1\\
 1 &1 &1
\end{pmatrix}.
$$
Then $r(A_4,\alpha)$ is an irreducible based module over $r(S_4)$ equivalent to $M_{3,1}$ listed in TABLE~\ref{t2}. In other words, the irreducible based module $M_{3,1}$ can be categorified by ${\rm Rep}(A_4,\alpha)$. %by the module category ${\rm Rep}(A_4,\alpha)$ over ${\rm Rep}(S_4)$.
\end{proof}

\begin{theorem}
$r(S_4,\alpha)={\rm Gr}({\rm Rep}(S_4,\alpha))$ is an irreducible based module over $r(S_4)$ equivalent to the based module $M_{3,3}$ in TABLE~\ref{t2}.
\end{theorem}
\begin{proof}
Let $\alpha$ be a nontrivial 2-cocycle in $Z^2(S_4,\CC^*)$ (see e.g. \cite[Section 3.2.4]{Ra}).
%For the irreducible representations $1,V_{\psi},V_{\rho_1},V_{\rho_2}$ and $V_{\rho_3}$ of $S_4$ mentioned before,
By checking products of characters, we get the following tensor product rule in $r^\alpha(S_4)$.
\begin{align*}
&1 \otimes V_{\xi_i }= V_{\xi_{i} },\quad V_\psi \otimes V_{\xi_j }= V_{\xi_{3-j} },\quad V_\psi \otimes V_{\xi_3 }= V_{\xi_{3} };\\
&V_{\rho_1} \otimes V_{\xi_j }= V_{\xi_{3} },\quad V_{\rho_1} \otimes V_{\xi_3 }= V_{\xi_{1} }+ V_{\xi_{2} }+V_{\xi_{3} };\\
&V_{\rho_2} \otimes V_{\xi_j }= V_{\xi_{3-j} }+ V_{\xi_{3} },\quad V_{\rho_2} \otimes V_{\xi_3 }= V_{\rho_3} \otimes V_{\xi_3 }= V_{\xi_{1} }+ V_{\xi_{2} }+2V_{\xi_{3} };\\
&V_{\rho_3} \otimes V_{\xi_j }= V_{\xi_{j} }+ V_{\xi_{3} };
\end{align*}
where $i=1,2,3$, $j=1,2$. Thus, we get
$$
1 \mapsto E_3,\quad V_{\psi} \mapsto \begin{pmatrix}
 0 &1 &0\\
 1 &0 &0\\
 0 &0 &1
\end{pmatrix},\quad V_{\rho_1} \mapsto \begin{pmatrix}
 0 &0 &1\\
 0 &0 &1\\
 1 &1 &1
\end{pmatrix},\quad V_{\rho_2} \mapsto \begin{pmatrix}
 0 &1 &1\\
 1 &0 &1\\
 1 &1 &2
\end{pmatrix},\quad V_{\rho_3} \mapsto \begin{pmatrix}
 1 &0 &1\\
 0 &1 &1\\
 1 &1 &2
\end{pmatrix}.$$
Then $r(S_4,\alpha)$ is an irreducible based module over $r(S_4)$ equivalent to $M_{3,3}$ listed in TABLE~\ref{t2}. In other words, the irreducible based module $M_{3,3}$ over $r(S_4)$ can be categorified by ${\rm Rep}(S_4,\alpha)$.% by the module category ${\rm Rep}(S_4,\alpha)$ over ${\rm Rep}(S_4)$.
\end{proof}

In summary, we have the following classification theorem.
\begin{theorem}\label{th:classification}
The inequivalent irreducible based modules over $r(S_4)$ are $M_{1,1}$, ${\left \{  M_{2,i}\right \} } _{i=1,2,3}$, ${\left \{  M_{3,j}\right \} } _{j=1,2,3}$, ${\left \{  M_{4,s}\right \} } _{1\le s\le 7}$ and ${\left \{  M_{5,t}\right \} } _{t=1,2}$, among which $M_{1,1}$, ${\left \{  M_{2,i}\right \} } _{i=2,3}$, ${\left \{  M_{3,j}\right \} } _{j=1,2,3}$, ${\left \{  M_{4,s}\right \} } _{s=1,5,7}$ and ${\left \{  M_{5,t}\right \} } _{t=1,2}$ can be categorified by module categories over ${\rm Rep}(S_4)$; see TABLE~\ref{t5}.
\end{theorem}

\newpage

\begin{table}[h]
    % 插入长度为5pt的垂直空间（也可以是负数，缩进）
    \vspace{-10pt}
    \centering
    % 表名 前面为中文名/后面为英文名
    \captionsetup[table]{position=above}
    % label标签，用以引用本表时。例：autoref{num}
    \label{num}
    % 设置表格单元格的列宽
    \setlength{\tabcolsep}{8mm}{
    % 表示 三线表 有4列
    \scalebox{0.546}{
    \begin{tabular}{cccccc}
    % toprule表示三线表的顶部线
        \toprule
        & &$V_{\psi}$ &$V_{\rho _{1}}$ &$V_{\rho _{2}}$ & categorification \\
        % midrule 表示 三线表的 中部线
        \midrule
        % 合并三行1列，用空格代替，也可以用\multirow{}[]{}{} 来表示
        Rank 1& \quad$M_{1,1}$  &$1$  &$2$   &$3$ & ${\rm Rep}(\mathbb Z_1)$,\,${\rm Rep}(K_4,\alpha)$\\\\
        Rank 2& \quad$M_{2,1}$  &$\begin{pmatrix}  1& 0\\  0&1\end{pmatrix}$  &$\begin{pmatrix}  2& 0\\  0&2\end{pmatrix}$   &$\begin{pmatrix}  1& 2\\  2&1\end{pmatrix}$  & No %$\varprod$
        \\\\
        & \quad$M_{2,2}$  &$\begin{pmatrix}  0& 1\\  1&0\end{pmatrix}$  &$\begin{pmatrix}  1& 1\\  1&1\end{pmatrix}$   &$\begin{pmatrix}  2& 1\\  1&2\end{pmatrix}$ &
        ${\rm Rep}(\mathbb Z_2)$,\,${\rm Rep}(D_4,\alpha)$ \\\\
        & \quad$M_{2,3}$  &$\begin{pmatrix}  0& 1\\  1&0\end{pmatrix}$  &$\begin{pmatrix}  1& 1\\  1&1\end{pmatrix}$   &$\begin{pmatrix}  1& 2\\  2&1\end{pmatrix}$ &
        ${\rm Rep}(\mathbb Z_2)$,\,${\rm Rep}(D_4,\alpha)$ \\\\
        Rank 3& \quad$M_{3,1}$  &$\begin{pmatrix}  1& 0&0\\  0&1&0\\0& 0&1\\\end{pmatrix}$  &$\begin{pmatrix}  0& 1&1\\  1&0&1\\1& 1&0\\\end{pmatrix}$   &$\begin{pmatrix}  1& 1&1\\ 1& 1&1\\1& 1&1\\\end{pmatrix}$ & ${\rm Rep}(\mathbb Z_3)$,\,${\rm Rep}(A_4,\alpha)$\\\\
        & \quad$M_{3,2}$  &$\begin{pmatrix}  0& 1&0\\  1&0&0\\0& 0&1\\\end{pmatrix}$  &$\begin{pmatrix}  0& 0&1\\ 0& 0&1\\1& 1&1\\\end{pmatrix}$   &$\begin{pmatrix}  1& 0&1\\ 0& 1&1\\1& 1&2\\\end{pmatrix}$ & ${\rm Rep}(S_3)$\\\\
        & \quad$M_{3,3}$  &$\begin{pmatrix}  0& 1&0\\  1&0&0\\0& 0&1\\\end{pmatrix}$  &$\begin{pmatrix}  0& 0&1\\ 0& 0&1\\1& 1&1\\\end{pmatrix}$   &$\begin{pmatrix}  0& 1&1\\ 1& 0&1\\1& 1&2\\\end{pmatrix}$ & ${\rm Rep}(S_3)$,\,${\rm Rep}(S_4,\alpha)$\\\\
        Rank 4& \quad$M_{4,1}$  &$\begin{pmatrix}  1& 0&0&0\\  0&1&0&0\\0&0&1&0\\0&0& 0&1\\\end{pmatrix}$  &$\begin{pmatrix}  0& 0&1&1\\ 0&2&0&0\\1&0&0&1\\1&0&1&0\\\end{pmatrix}$   &$\begin{pmatrix}  0& 1&0&0\\  1&2&1&1\\0&1&0&0\\0&1&0&0\\\end{pmatrix}$  & ${\rm Rep}(A_4)$\\\\
        & \quad$M_{4,2}$  &$\begin{pmatrix}  1& 0&0&0\\  0&1&0&0\\0&0&1&0\\0&0& 0&1\\\end{pmatrix}$  &$\begin{pmatrix}  2& 0&0&0\\  0&2&0&0\\0&0&2&0\\0&0&0&2\\\end{pmatrix}$   &$\begin{pmatrix}  0& 1&1&1\\  1&0&1&1\\1&1&0&1\\1&1&1&0\\\end{pmatrix}$ & No\\\\
        & \quad$M_{4,3}$  &$\begin{pmatrix}  1& 0&0&0\\  0&1&0&0\\0&0&1&0\\0&0& 0&1\\\end{pmatrix}$  &$\begin{pmatrix}  1& 1&0&0\\  1&1&0&0\\0&0&2&0\\0&0&0&2\\\end{pmatrix}$   &$\begin{pmatrix}  0& 1&1&1\\  1&0&1&1\\1&1&0&1\\1&1&1&0\\\end{pmatrix}$ & No\\\\
        & \quad$M_{4,4}$  &$\begin{pmatrix}  0& 1&0&0\\  1&0&0&0\\0&0&1&0\\0&0& 0&1\\\end{pmatrix}$  &$\begin{pmatrix}  1& 1&0&0\\  1&1&0&0\\0&0&2&0\\0&0&0&2\\\end{pmatrix}$   &$\begin{pmatrix}  1&0&1&1\\  0&1&1&1\\1&1&0&1\\1&1&1&0\\\end{pmatrix}$ & No\\\\
        & \quad$M_{4,5}$  &$\begin{pmatrix}  0& 1&0&0\\  1&0&0&0\\0&0&0&1\\0&0& 1&0\\\end{pmatrix}$  &$\begin{pmatrix}  1& 1&0&0\\  1&1&0&0\\0&0&1&1\\0&0&1&1\\\end{pmatrix}$   &$\begin{pmatrix}  0&1&1&1\\  1&0&1&1\\1&1&0&1\\1&1&1&0\\\end{pmatrix}$ & ${\rm Rep}(\mathbb Z_4),\,{\rm Rep}(K_4)$ \\\\
        & \quad$M_{4,6}$  &$\begin{pmatrix}  0& 1&0&0\\  1&0&0&0\\0&0&0&1\\0&0& 1&0\\\end{pmatrix}$  &$\begin{pmatrix}  1& 1&0&0\\  1&1&0&0\\0&0&1&1\\0&0&1&1\\\end{pmatrix}$   &$\begin{pmatrix}  0&1&1&1\\  1&0&1&1\\1&1&1&0\\1&1&0&1\\\end{pmatrix}$ & No\\\\
        & \quad$M_{4,7}$  &$\begin{pmatrix}  0& 1&0&0\\  1&0&0&0\\0&0&0&1\\0&0& 1&0\\\end{pmatrix}$  &$\begin{pmatrix}  1& 1&0&0\\  1&1&0&0\\0&0&1&1\\0&0&1&1\\\end{pmatrix}$   &$\begin{pmatrix}  1&0&1&1\\  0&1&1&1\\1&1&1&0\\1&1&0&1\\\end{pmatrix}$ & ${\rm Rep}(\mathbb Z_4),\,{\rm Rep}(K_4)$ \\\\
        Rank 5& \quad$M_{5,1}$  &$\begin{pmatrix}  0& 1&0&0&0\\  1&0&0&0&0\\0&0&0&1&0\\0&0& 1&0&0\\0&0& 0&0&1\\\end{pmatrix}$  &$\begin{pmatrix}  0& 0&0&0&1\\  0& 0&0&0&1\\0&0&1&1&0\\0&0& 1&1&0\\1&1& 0&0&1\\\end{pmatrix}$   &$\begin{pmatrix}  0& 0&0&1&0\\  0&0&1&0&0\\0&1&1&1&1\\1&0& 1&1&1\\0&0& 1&1&0\\\end{pmatrix}$ & ${\rm Rep}(S_4)$ \\\\
        & \quad$M_{5,2}$  &$\begin{pmatrix}  0& 1&0&0&0\\  1&0&0&0&0\\0&0&0&1&0\\0&0& 1&0&0\\0&0& 0&0&1\\\end{pmatrix}$  &$\begin{pmatrix}  1& 1&0&0&0\\  1& 1&0&0&0\\0&0&1&1&0\\0&0& 1&1&0\\0&0& 0&0&2\\\end{pmatrix}$   &$\begin{pmatrix}  0& 0&0&1&1\\  0& 0&1&0&1\\0&1&0&0&1\\1&0& 0&0&1\\1&1& 1&1&1\\\end{pmatrix}$ & ${\rm Rep}(D_4)$ \\[1em]
        % bottomrule表示 三线表 的底部线
        \bottomrule
    \end{tabular}}}
     \vspace{.5em}
     \caption{Inequivalent irreducible based modules over $r(S_4)$}\label{t5}
\end{table}

\noindent{\textbf{Open problem.}  Classifying irreducible $\mathbb Z_+$-modules over $r(S_4)$ is more challenging, especially for high-rank cases, and it might be considered in a sequel of this paper. Additionally, the existence of any irreducible based module of rank $\ge$ 6 over $r(S_4)$ warrants further investigation.

\medskip\noindent
{\bf Acknowledgments.} We would like to thank Zhiqiang Yu for helpful discussion. This work is supported by Natural Science Foundation of China (12071094, 12171155) and Guangdong Basic and Applied Basic Research Foundation (2022A1515010357).

\bibliographystyle{amsplain}

\end{document}